\documentclass[sn-mathphys]{sn-jnl}% Math and Physical Sciences Reference Style
\usepackage{epstopdf,varwidth}
\jyear{2021}%

\theoremstyle{thmstyleone}%
\newtheorem{theorem}{Theorem}%  meant for continuous numbers
\newtheorem{assumption}[theorem]{Assumption}%
\newtheorem{lemma}[theorem]{Lemma}%
\newtheorem{corollary}[theorem]{Corollary}%

\theoremstyle{thmstyletwo}%

\theoremstyle{thmstylethree}%

\newcommand{\N}{\mathbb{N}}
\newcommand{\E}{\mathbb{E}}
\renewcommand{\P}{\mathbb{P}}

\newcommand{\dd}{\text{d}}

\begin{document}

\title[An explicit approximation for SCHE with additive noise]{Strong convergence rates of an explicit scheme for stochastic Cahn--Hilliard equation with additive noise}


\author[1,2]{\fnm{Meng} \sur{Cai}}\email{mcai@lsec.cc.ac.cn}

\author[3]{\fnm{Ruisheng} \sur{Qi}}\email{qiruisheng123@126.com}

\author*[1]{\fnm{Xiaojie} \sur{Wang}}\email{x.j.wang7@csu.edu.cn; x.j.wang7@gmail.com}

\affil[1]{\orgdiv{School of Mathematics and Statistics, HNP-LAMA}, \orgname{Central South University}, \orgaddress{\city{Changsha}, \postcode{410083}, \country{China}}}

\affil[2]{\orgdiv{LSEC, ICMSEC, Academy of Mathematics and Systems Science}, \orgname{Chinese Academy of
Sciences}, \orgaddress{\city{Beijing}, \postcode{100190}, \country{China}}}

\affil[3]{\orgdiv{School of Mathematics and Statistics}, \orgname{Yancheng Teachers University}, \orgaddress{ \city{Yancheng}, \postcode{224002}, \country{China}}}


\abstract{In this paper, we propose and analyze an explicit time-stepping scheme for a spatial discretization of
stochastic Cahn--Hilliard equation with additive noise.
The fully discrete approximation combines a spectral Galerkin method in space with a tamed exponential Euler method in time.
In contrast to implicit schemes in the literature, the explicit scheme here is easily implementable
and produces significant improvement in the computational efficiency.
It is shown that the fully discrete approximation  converges strongly to the exact solution, with strong convergence rates identified.
Different from the tamed time-stepping schemes for stochastic Allen--Cahn  equations,
essential difficulties arise in the analysis due to the presence of the unbounded linear operator in front of the nonlinearity.
To overcome them, new and non-trivial arguments are developed in the present work.
To the best of our knowledge, it is the first result concerning an explicit scheme for the stochastic Cahn--Hilliard equation.
Numerical experiments are finally performed to confirm the theoretical results.}




\maketitle

\section{Introduction}

Let  $\mathcal{D} $ be a bounded convex domain in $\mathbb{R}^d, d \in \{1,2,3\}$.
%$\mathcal{D}=(0,1)^d$ with $d \in \{1,2,3\}$
We denote by $ H=L^2 ( \mathcal{D},\mathbb{R} ) $ a real separable Hilbert space with scalar product $ \langle \cdot,\cdot \rangle$
and norm $\|\cdot\|$ and
$\dot{H} := \{ v \in H: \int_{\mathcal{D}} v \dd x =0 \}$.
In this paper, we consider the numerical approximation of
stochastic Cahn--Hilliard equation (SCHE) in the abstract form
\begin{equation}\label{eq:CHC-abstract}
\begin{split}
\left\{
    \begin{array}{lll}
    \dd X(t) + A(A X(t)+F ( X(t) ))
    \, \dd t
    =
    \dd W(t), \quad  t \in (0, T],
    \\
     X(0) = X_0,
    \end{array}\right.
\end{split}
\end{equation}
where $0<T<\infty$, $-A$ is the Neumann Laplacian and
$\{W(t)\}_{t \geq 0}$ is a $Q$-Wiener process on a filtered probability space
$\big( \Omega, \mathcal {F}, \P, \{\mathcal {F}_t\}_{t \geq 0} \big)$, specified later.
The nonlinear term $F$ is assumed to be a Nemytskii operator, given by
$F(u)(x) = f (u(x)) = u^3(x) - u(x), x \in \mathcal{D}$.
%It is easy to check that $F$ satisfies the monotonous condition
%$  - \langle F(u)-F(v) , u-v \rangle  \leq  \| u-v \|^2$ for
%$ u, v \in L^6(\mathcal{D},\R)  $.
As a phenomenological model from metallurgy and physics,
the deterministic version of such equation is used to describe the complicated phase separation and coarsening phenomena in a melted alloy
\cite{cahn1961on,cahn1971spinodal}
and spinodal decomposition for binary mixture
\cite{cahn1958free}.
Adding a noise to the physical model is quite natural as it either represents
an external random perturbation or gives a remedy for lack of knowledge of certain involved physical parameters.
For example, in \cite{blomker2001spinodal,cook1970brownian} and references therein, the authors have expressed the belief that only the stochastic version can correctly describe the whole decomposition process in
a binary alloy.
The stochastic version \eqref{eq:CHC-abstract} has been extensively studied by many authors (see e.g., \cite{antonopoulou2016existence,cui2019wellposedness,
cui2020absolute,prato1996stochastic,blomker2001spinodal, elezovic1991on,qi2019sharp,kovacs2011finite,qi2020error,furihata2018strong}).

Since the true solution of the  problem can not be known explicitly,
it is therefore natural to look for reliable numerical solutions.
To do the approximation error analysis, one often faces difficulties, raised by the presence of the unbounded operator $A$ in front of the nonlinear term $F$.
In the past few years, many authors investigated strong and weak approximations of stochastic Cahn--Hilliard equation
\cite{chai2018conforming, kossioris2013finite, larsson2011finite,furihata2018strong,cai2021weakSCHE,cui2021strongCHC,
hong2022convergence,hong2022finite,qi2020error,feng2020fully},
%The unbounded operator $A$ in front of the nonlinear term
%$F$ raises many new difficulties and makes the error analysis much more demanding.
where some attempts to address the issue were proposed in literature.
For the linearized stochastic Cahn--Hilliard equation,
the readers are referred to \cite{chai2018conforming, kossioris2013finite, larsson2011finite}.
In \cite{hutzenthaler2020perturbation},  strong convergence rates
for the spectral Galerkin spatial approximation of the nonlinear problem in dimension one
are proved by combining a general perturbation theory
with the exponential integrability properties of the numerical approximation.
The authors in \cite{furihata2018strong,kovacs2011finite} derive the strong convergence of the finite element spatial approximation
and the backward Euler full discretization of  the SCHE driven by spatial regular noise, but with no rates obtained.
Very recently, the paper \cite{qi2020error} fills the gap left by \cite{furihata2018strong,kovacs2011finite}
and recovers the strong convergence rates of the finite element fully discrete scheme.
For space-time white noise, authors in \cite{cui2021strongCHC} obtained the strong convergence rates of a fully discrete scheme performed by a spatial spectral Galerkin method and a temporal  accelerated implicit Euler method.
Moreover, the strong convergence rates of an implicit fully discrete mixed finite element method
for the SCHE with gradient-type multiplicative noise
are derived in \cite{feng2020fully}, where the noise process is a real-valued Wiener process.
To the best of our knowledge,
the explicit methods are absent for the SCHE and this paper aims to propose an explicit  scheme for the equation
and identify its strong convergence rates.

As indicated in \cite{beccari2019strong}, the fully discrete exponential Euler and fully discrete linear-implicit Euler approximations
diverge strongly and numerically weakly in the case of stochastic Allen--Cahn equations.
Later, some explicit modified Euler-type schemes have been proposed in \cite{becker2017strong,brehier2020approximation,
cai2021weak,gyongy2016convergence,Wang2020efficient} to numerically solve the stochastic Allen--Cahn equations.
Based on a spectral Galerkin spatial  approximation of \eqref{eq:CHC-abstract}, given by
%on vector space $H_N := {\rm span} \{e_1 , e_2 , \cdots , e_{N}\}$,
%in the form
\begin{equation*}
\begin{split}
\left\{
    \begin{array}{lll}
    \dd X^N(t) + A (A X^N(t)  + P_N F( X^N(t) )) \dd t
      = P_N \dd W(t), \quad  t \in (0, T],
    \\
     X^N(0)=P_N X_0,
    \end{array}\right.
\end{split}
\end{equation*}
we propose a tamed exponential Euler scheme in time to obtain the explicit fully discrete method
\begin{equation*}
X_{t_{m+1}}^{M,N}  =  E(\tau) X_{t_m}^{M,N}
%-\tfrac{ A^{-1}(I-E(\tau)) P_N F (X_{t_m}^{M,N})}
%{1+\tau \| P_N F (X_{t_m}^{M,N}) \|}
- \int_{t_m}^{t_{m+1}}
\tfrac{E(t_{m+1}-s)AP_NF(X_{t_m}^{M,N})}
{1+\tau \| P_N F (X_{t_m}^{M,N}) \|} \dd s
+ E (\tau) P_N \Delta W_m,
\end{equation*}
where $\Delta W_m = W(t_{m+1})-W(t_m),\, m \in \{0,1,2,\cdots,M-1\}$, $P_N $ is the projection operator onto
$H_N := {\rm span} \{e_1 , e_2 , \cdots , e_{N}\}$,
$E(t)=e^{-tA^2},t \ge 0$ denotes an analytic semigroup on $H$  generated by $-A^2$
 and $\tau=\tfrac TM$ stands for the time step-size.
Compared with existing implicit schemes, the proposed scheme is easy to implement and produces significant improvement in the computational efficiency.

Throughout this article, $C$ denotes a generic nonnegative constant that is independent of the discretization parameters and may change from line to line.
Meanwhile, we use $\N^{+}$ to denote the set of all positive integers and
$\N = \{ 0 \} \cup \N^{+}$.
In summary, the contribution of this article to the numerical analysis of stochastic Cahn--Hilliard equation is twofold.
On the one hand, the uniform a priori moment bounds of the full discretization are derived based on a certain bootstrap argument.
To do this, a key ingredient lies on  bounding
\[
\sup_{ M,N \in \N^{+} } \sup_{ m \in \{ 0,1,\ldots,M \} }
 \E \Big[ \| X_{t_m}^{M,N} \|_{L^6}^p \Big] <\infty
\]
by virtue of Gagliardo--Nirenberg inequality in $d=1$ and energy estimate in $d=2,3$.
On the other hand, as implied by Corollary \ref{coro:strong full discretization},
we identify the strong convergence rate of the fully discrete method:
\begin{equation}\label{eq:intro-convergence}
\sup_{ M,N \in \N^{+} } \sup_{ m \in \{0,1,\ldots,M\} }
  \| X(t_m) - X_{t_m}^{M,N} \|_{ L^p (\Omega , \dot{H}) }
    \leq C \,( \lambda_N^{-\frac{\gamma}{2}} + \tau^{ \frac{\gamma}{4} }),
    \, \gamma \in (\tfrac d2,4],
\end{equation}
where $\lambda_N$ is the  $N$-th eigenvalue of $A$ and $\gamma$ from Assumption \ref{assum:noise-term} is a parameter used to measure the spatial regularity of the noise process.
The above result reveals that the strong convergence rates are essentially governed by the spatial regularity of the noise term.
The rates of convergence are optimal.
Comparing \eqref{eq:intro-convergence} with the sharp temporal H\"{o}lder regularity result in Theorem \ref{thm:uniqueness-mild-solution},
one can easily observe, for $\gamma \in (\tfrac d2, 2]$, the rate of convergence is in accordance with the temporal H\"{o}lder regularity of the mild solution.
%and the obtained convergence rate is then called optimal in this sense.
For $\gamma \in [2, 4]$, the rate of the convergence can reach $1$ and higher than the H\"{o}lder continuity of the mild solution due to the fact that the noise is additive.
The convergence rates are the same as that obtained for the backward Euler method from the literature \cite{qi2019sharp,qi2020error}.
It must be emphasized that the derivation of \eqref{eq:intro-convergence} is not an easy task and requires a variety of delicate error estimates, which are elaborated in subsection \ref{sec;strong convergence}.

The outline of the article is organized as follows.
In the next section, we present some assumptions
and give the well-posedness and regularity of the mild solution.
Section \ref{sec:full} is devoted to the strong convergence analysis, where spectral Galerkin method is introduced in subsection \ref{subsec:Galerkin}, uniform a priori moment bounds are deduced in subsection \ref{subsec:moment bounds} and the strong convergence rates are derived in subsection \ref{sec;strong convergence}.
Numerical examples are finally included in Section \ref{sec:numerical-experiments} to verify the theoretical findings.

\section{Main assumptions and the considered problem}\label{sec:preliminaries}

Given another separable Hilbert space
$(U, \langle \cdot, \cdot \rangle_U, \|\cdot\|_U )$,
$\mathcal{L}(U,H)$ represents the space of all bounded linear operators from $U$ to $H$ endowed with the usual operator norm
$\| \cdot \|_{\mathcal{L}(U,H)}$
and by $\mathcal{L}_2(U,H) \subset \mathcal{L}(U,H)$
we denote the space consisting of all Hilbert--Schmidt operators from $U$ to $H$.
To simplify the notation, we often write $\mathcal{L}(H) $ and $\mathcal{L}_2(H)$ (or $\mathcal{L}_2$ for short) instead of $\mathcal{L}(H,H)$ and $\mathcal{L}_2(H,H)$, respectively.
It is easy to prove that $\mathcal{L}_2 (U,H)$ is a Hilbert space equipped with the inner product and norm,
\begin{align*}
\langle T_1 , T_2 \rangle_{\mathcal{L}_2(U,H)}
=
\sum_{i\in\N^{+}} \langle T_1 \phi_i , T_2 \phi_i \rangle,
\;
\| T \|_{\mathcal{L}_2(U,H)}
=\Big(
\sum_{i \in \N^{+}} \| T \phi_i \|^2
\Big)^{\frac12},
\end{align*}
independent of the choice of orthonormal basis $\{\phi_i\}$ of $U$.
If $ T \in \mathcal{L}_2(U,H)$ and $L\in \mathcal{L}(H,U)$,
then $ T L\in \mathcal{L}_2(H)$ and
$ \| T L \|_{\mathcal{L}_2(H)} \leq
       \| T \|_{\mathcal{L}_2(U,H)}  \|L\|_{\mathcal{L}(H,U)}. $
Also,  $  \vert \langle T_1, T_2  \rangle_{\mathcal{L}_2(U,H)} \vert \leq \| T_1 \|_{\mathcal{L}_2(U,H)} \| T_2  \|_{\mathcal{L}_2(U,H)} $ holds for $T_1, T_2 \in \mathcal{L}_2(U,H)$.
Finally,  $V:=C(\mathcal{D},\mathbb{R})$ represents the Banach space of all continuous functions from $\mathcal{D}$ to $\mathbb{R}$ with supremum norm.
Throughout this paper, we define  an orthogonal projector $P:H \rightarrow \dot{H}$ by
\begin{equation*}
P v= v - \vert \mathcal D \vert^{-1} \int_{\mathcal D} v \dd x
\end{equation*}
and then
$(I-P)v= \vert \mathcal D \vert^{-1} \int_{\mathcal D} v \dd x$
is the average of $v$.
Here and below, by
$ L^r(\mathcal{D} , \mathbb{R}), r \geq 1 $
($L^r(\mathcal{D})$  or $L^r$ for short)
we denote a Banach space consisting of all $r$-times integrable functions.

In the sequel, the main assumptions are made for the abstract model \eqref{eq:CHC-abstract}.
\begin{assumption}\label{assum:linear-operator-A}
Let $\mathcal{D} $ be a bounded convex domain of $\mathbb{R}^d, d \in \{1,2,3\}$ with Lipschitz boundary.
Let $-A$ be the Neumann Laplacian, given by $-Au = \Delta u$,
with $ u \in \mathrm{dom}(A) := \{ v\in H^2(\mathcal D) \cap \dot{H}:
   \tfrac{\partial v}{\partial n}=0  \,\, on \,\, \partial \mathcal D\} $.
\end{assumption}

For $ v \in H$, we extend the definition as $A v=A P v$.
Then there exists a family of orthonormal eigenbasis
$\{e_j\}_{j \in \N }$
with corresponding eigenvalues
$\{\lambda_j\}_{j \in \N}$
such that
\begin{align*}
A e_j=\lambda_j e_j,
\,\,
0= \lambda_0 < \lambda_1 \leq \lambda_2 \leq \cdots
  \leq \lambda_j \leq \cdots, \,\,
\lambda_j \rightarrow \infty \,\, \text{as} \,\, j \to \infty,
\end{align*}
where
$e_0= \vert \mathcal D \vert^{-\frac12}$
and
$\{ e_j \}_{ j \in \N^{+} }$
forms an orthonormal basis of $\dot{H}$.
We define the fractional powers of $A$ on $\dot{H}$ by the spectral theory,
e.g.,
$A^\alpha v = \sum_{j=1}^\infty \lambda_j^\alpha \langle v,e_j \rangle e_j$,
$\alpha \in \mathbb{R}$.
The space $\dot{H}^\alpha := \mathrm{dom} ( A^{\frac \alpha 2 } )$ is a Hilbert space  with the inner product $\langle \cdot, \cdot \rangle_{\alpha}$
and the associated norm $\vert \cdot \vert_\alpha $ defined by
\begin{equation*}
\langle v , w \rangle_{\alpha}
   =  \sum_{j=1}^{\infty} \lambda_j^{\alpha}
      \langle v , e_j \rangle \langle w , e_j \rangle,
\,\, \vert v \vert_{\alpha}=  \| A^{\frac\alpha 2} v \|=
\Big( \sum_{j=1}^{\infty} \lambda_j^{\alpha}
        \vert \langle v,e_j\rangle \vert^2   \Big)^{\frac12},
        \, \alpha \in \mathbb{R}.
\end{equation*}
%Furthermore, we introduce $H^{\alpha}:=\{v \in H:\|v\|_{\alpha} < \infty \}$ with $\|v\|_{\alpha}
%=\big( |v|_{\alpha}^2+|\langle v,e_0\rangle|^2\big)^{\frac12}$.
Note that for integer $k \geq 0$,
$\dot{H}^k$
is a subspace of
$H^k (D) \cap \dot{H}$
characterized by certain boundary conditions.
Let us recall the following results concerning the spaces $\dot{H}^{\alpha}$ and the associated norms $\vert \cdot \vert_\alpha $  for $\alpha \in [0,2]$, see
 \cite{furihata2018strong,kim2020fractional,yagi2010abstract} for more details.
For $\alpha \in [0,\tfrac32)$, one has
$$\dot{H}^{\alpha}=H^\alpha(\mathcal D),$$
and for  $\alpha \in (\tfrac32,2]$, one has
$$\dot{H}^{\alpha}=H_N^\alpha(\mathcal D)
:=\{v \in H^\alpha(\mathcal D): \tfrac{\partial v}{\partial n}=0 \,\, \textrm{on} \,\, \partial \mathcal D \} \subset H^\alpha(\mathcal D).$$
Additionally, for $\alpha \in [0,\tfrac32)\cup(\tfrac32,2]$,
the norm
$\vert \cdot \vert_\alpha$
is  equivalent on $\dot{H}^\alpha$ to the standard Sobolev norm
$\| \cdot \|_{H^\alpha(\mathcal D)}$.
Since $H^2(\mathcal D)$ is an algebra,
one can deduce that for any $f,g \in \dot{H}^2$,
 \begin{equation}\label{eq:norm-algebra}
 \|fg\|_{H^2(\mathcal D)} \leq
      C \|f\|_{H^2(\mathcal D)} \|g\|_{H^2(\mathcal D)}
     \leq  C \vert f \vert_2 \vert g \vert_2.
 \end{equation}
Additionally, the operator $-A^2$ generates an analytic semigroup $E(t)=e^{-tA^2}$ on $H$, given by
\begin{align*}
\begin{split}
E(t)v=e^{-tA^2}v
&=
\sum_{j=0}^\infty e^{-t \lambda_j^2}\big\langle v,e_j\big\rangle e_j
=
\sum_{j=1}^\infty e^{-t \lambda_j^2}\big\langle v,e_j\big\rangle e_j
+
\big\langle v,e_0\big\rangle e_0
\\
&=Pe^{-t A^2}v + (I-P)v, \quad v \in H.
\end{split}
\end{align*}
At last, the properties of $E(t)$ are obtained by expansion in terms of the eigenbasis of $A$ and using Parseval's identity,
\begin{align}
&\| A^\mu E(t) \|_{\mathcal{L}(\dot{H})}
\leq
 C t^{-\frac\mu2}, \, t>0, \, \mu \geq 0,
 \label{I-spatial-temporal-S(t)}
 \\
 & \| A^{-\nu}(I-E(t)) \|_{\mathcal{L}(\dot{H})}
 \leq
 C t^{\frac\nu2}, \, t \geq 0, \, \nu \in[0,2],
 \label{II-spatial-temporal-S(t)}
 \\   &
\int_{t_1}^{t_2} \! \! \| A^\varrho E(s) v \|^2  \dd s
 \leq C \vert t_2-t_1 \vert^{1-\varrho} \| v \|^2, \, \forall v \in \dot{H},
\varrho \in [0,1], 0 \leq t_1 \leq t_2,
\label{III-spatial-temporal-S(t)}
\\ &
 \Big\| A^{2\rho} \! \! \! \int_{t_1}^{t_2} \! \! \!
         E(t_2-\sigma)v  \dd \sigma
 \Big\|
 \leq
 C \vert t_2 \! - \! t_1 \vert^{1-\rho} \| v \|, \, \forall v \in \dot{H},
 \rho \in [0,1], 0 \leq t_1 \leq t_2.
 \label{IV-spatial-temporal-S(t)}
 \end{align}

\begin{assumption}\label{assum:nonlinearity}
Let
$F : L^6 ( \mathcal{D} , \mathbb{R} ) \rightarrow H$
be the Nemytskii operator given by
\begin{align*}
F (v)(x)
=f ( v(x))
=v^3(x)-v(x),
\quad
x \in \mathcal{D}, v \in L^6(\mathcal{D} , \mathbb{R}).
\end{align*}
\end{assumption}
Then, for $v, \zeta,  \zeta_1, \zeta_2 \in L^6(\mathcal{D},\mathbb{R})$,
we have
\begin{equation*}
\begin{split}
\big ( F'(v) (\zeta) \big) (x)
& =
f'(v(x)) \zeta ( x )
=
(
3 v^2 ( x )-1
)
\zeta ( x ),
\quad
x\in \mathcal{D},
 \\
 \big( F''(v) ( \zeta_1, \zeta_2 ) \big ) (x)
 & =
 f''(v(x))  \zeta_1 ( x ) \zeta_2 ( x )
 =
 6 v(x) \zeta_1 ( x ) \zeta_2 ( x ),
 \quad
 x\in \mathcal{D},
\end{split}
\end{equation*}
where the above derivatives can be understood as Gateaux derivatives in Banach spaces.
As a result, there exists a constant $C > 0$ such that
\begin{equation}\label{eq:one-side-condition}
- \langle F(u)-F(v) , u-v \rangle
   \leq  \| u-v \|^2, \quad u, v \in L^6(\mathcal{D}).
\end{equation}
\begin{equation}\label{eq:F'-condition}
\|F'(v)u\| \leq C \big( 1+ \| v \|_V^2 \big)\|u\|,
\quad u,v \in V.
\end{equation}
\begin{equation}\label{eq:local-condition}
\| F (u) - F (v) \|  \leq
C ( 1 + \| u \|_V^2 +  \| v \|_V^2 ) \| u - v \|,
\quad u, v \in V.
\end{equation}
To simplify the presentation, we assume the average of the Wiener process to be zero
so that the covariance operator $Q$ of the $Q$-Wiener process  belongs to $ \mathcal{L}(\dot{H}) $.
\begin{assumption}\label{assum:noise-term}
Let $\{ W(t)\}_{t \in [0,T]}$ be a  $\dot{H}$-valued (possibly cylindrical) Q-Wiener process
with the covariance operator $Q \in \mathcal{L}(\dot{H})$ satisfying
\begin{align}\label{eq:ass-AQ-condition}
\Big\| A^{\frac{\gamma-2}2} Q^{\frac12} \Big\|_{\mathcal{L}_2}
  < \infty \quad \text{for some} \quad
   \gamma \in \big( \tfrac{d}{2} , 4 \big].
\end{align}
\end{assumption}

\begin{assumption}\label{assum:intial-value-data}
Let $X_0: \Omega \rightarrow \dot{H}$ be
$\mathcal{F}_0 / \mathcal{B}(\dot{H})$-measurable
and satisfy that for a sufficiently large number $ p_0 \in \mathbb{N} $,
\begin{align*}
\E [ \vert X_0 \vert_\gamma^{p_0} ] <\infty,
\end{align*}
where $\gamma$ is the parameter from \eqref{eq:ass-AQ-condition}.
\end{assumption}

Before moving on, similar to \cite[Equations (2.5), (2.7)]{cui2021strongCHC}, we give the following lemma concerning the spatio-temporal regularity results of stochastic convolution
\begin{equation*}
\mathcal{O}_t := \int_0^t  E(t-s)  \dd W(s).
\end{equation*}

\begin{lemma}\label{lem:stochastic-convolution}
Suppose Assumptions \ref{assum:linear-operator-A} and \ref{assum:noise-term} hold.
Then for any $p \ge 1$,
the stochastic convolution $\mathcal{O}_t$ satisfies
\begin{equation*}
\E \Big[ \sup_{t\in[0,T]} \| \mathcal{O}_t \|_{V}^p \Big]
 + \sup_{t\in[0,T]} \E \Big[ \vert \mathcal{O}_t \vert_{\gamma}^p \Big] < \infty,
\end{equation*}
and for $\alpha \in [0,\gamma]$ and $0 \leq s \leq  t \leq T$,
\begin{equation*}
\| \mathcal{O}_t - \mathcal{O}_s \|_{L^p( \Omega , \dot{H}^{\alpha} )}
  \leq C ( t-s )^{\mathrm{min} \{  \frac12,\frac{\gamma-\alpha}{4}  \}}.
\end{equation*}
\end{lemma}

\begin{proof}
Applying the  factorization method used in
\cite[Theorem 5.10]{prato2014stochastic} yields that,
for $\alpha \in (0,1)$,
\begin{equation*}
\mathcal{O}_t = \frac{\sin(\alpha \pi)}{\pi}
  \int_0^t ( t - s )^{\alpha - 1} E( t - s ) Y_{\alpha} (s) \dd s,
\end{equation*}
where
\begin{equation*}
 Y_{\alpha} (s) := \int_0^s ( s - r )^{-\alpha}
                   E (s -  r )  \dd W(r).
\end{equation*}
Then, by the Burkholder--Davis--Gundy inequality and H\"older's inequality,  for a sufficiently large $p>1$ and
$\tfrac d2 < \theta < \min \{ \gamma, 2\}$ such that
$\tfrac 1p + \tfrac \theta4 < \alpha < \min \{ \tfrac \gamma 4, \tfrac 12 \}$,
we obtain
\begin{equation*}
\begin{split}
\E \Big[  \sup_{t\in[0,T]} \vert \mathcal{O}_t \vert_{\theta}^p \Big]
 &= \E \Big[ \sup_{t\in[0,T]} \Big\vert \frac{\sin(\alpha \pi)}{\pi}
  \int_0^t ( t - s )^{\alpha - 1} E( t - s ) Y_{\alpha} (s) \dd s \Big\vert_{\theta}^p \Big]
  \\ & \leq C \, \E \Big[ \sup_{t\in[0,T]} \Big(
  \int_0^t ( t - s )^{\alpha - 1 -\frac \theta 4} \vert Y_{\alpha} (s) \vert \dd s \Big)^p \Big]
  \\ & \leq C \, \E \Big[ \sup_{t\in[0,T]} \Big(
  \int_0^t ( t - s )^{(\alpha - 1 -\frac \theta 4)q} \dd s \Big)^\frac pq
  \cdot \int_0^t \vert Y_{\alpha} (s) \vert^p \dd s  \Big]
  \\ & \leq C \int_0^T \E \big[ \vert Y_{\alpha} (s) \vert^p \big] \dd s
   \\ & \leq C \int_0^T \Big( \int_0^s ( s - r )^{ - 2 \alpha }
        \| E ( s- r) Q^{\frac 12}
            \|_{\mathcal{L}_2}^2 \dd r \Big)^{\frac p2} \dd s
   \\ & \leq C \int_0^T \Big( \int_0^s
          ( s - r )^{ - 2 \alpha + \min \{ \frac {\gamma-2}2,0\} }
        \| A^{\frac {\gamma-2}2} Q^{\frac 12}
            \|_{\mathcal{L}_2}^2 \dd r \Big)^{\frac p2} \dd s
   \\ & < \infty.
\end{split}
\end{equation*}
With aid of the Sobolev embedding inequality
$\dot{H}^{\theta} \subset V,\theta > \tfrac d2$,
we arrive at
\begin{equation*}
\E \Big[ \sup_{t\in[0,T]} \| \mathcal{O}_t \|_{V}^p \Big] < \infty.
\end{equation*}
By means of the Burkholder--Davis--Gundy--type inequality, \eqref{III-spatial-temporal-S(t)} and \eqref{eq:ass-AQ-condition},
we have
\begin{equation*}
\begin{split}
\| \mathcal{O}_t \|_{L^p(\Omega,\dot{H}^{\gamma})}
   & \leq  C \Big(\int_0^t \big\| A^{\frac \gamma 2}
  E (t-s) \big\|_{\mathcal{L}_2^0}^2  \dd s  \Big)^{\frac 12}
   \\ & = C \Big(\int_0^t \big\| A E (t-s)
    A^{\frac {\gamma-2} 2} \big\|_{\mathcal{L}_2^0}^2  \dd s
   \Big)^{\frac 12}
   \\ & \leq C \| A^{\frac {\gamma-2} 2}
               Q^{\frac 12} \|_{\mathcal{L}_2} < \infty.
\end{split}
\end{equation*}
Similarly, for $\alpha \in [0,\gamma]$,
\begin{equation*}
\begin{split}
\| & \mathcal{O}_t - \mathcal{O}_s \|_{L^p(\Omega,\dot{H}^{\alpha})}
 \\ &   \leq  C \Big( \int_0^s \big\| A^{\frac \alpha 2}
 ( E (t-r) - E (s-r) ) \big\|_{\mathcal{L}_2^0}^2
 \dd r  \Big)^{\frac 12}
  + C \Big( \int_s^t \big\| A^{\frac \alpha 2}
  E (t-r)  \big\|_{\mathcal{L}_2^0}^2 \dd r  \Big)^{\frac 12}
   \\ & = C \Big( \int_0^s \big\| A  E (s-r)
 A^{\frac {\alpha - \gamma} 2}
 ( E (t-s) -I ) A^{\frac  {\gamma-2} 2} \big\|_{\mathcal{L}_2^0}^2
 \dd r  \Big)^{\frac 12}
 \\ & \quad + C \Big( \int_s^t \big\| A^{\frac { 2 - \gamma+ \alpha} 2}
  E (t-r) A^{\frac  {\gamma-2} 2} \big\|_{\mathcal{L}_2^0}^2
  \dd r  \Big)^{\frac 12}
 \\ & \leq C ( t - s )^{\frac{\gamma-\alpha}{4}}
 + C ( t-s )^{\mathrm{min} \{  \frac12,\frac{\gamma-\alpha}{4}  \}}
 \\ & \leq C ( t-s )^{\mathrm{min}
       \{  \frac12,\frac{\gamma-\alpha}{4}  \}}.
\end{split}
\end{equation*}
Hence, we complete the proof.
\end{proof}

At last, we consider the mild solution of \eqref{eq:CHC-abstract} by following a semigroup approach proposed in \cite{prato2014stochastic}.
As already proved in
\cite[Proposition 6 $\&$ Proposition 7]{cui2021strongCHC},
the above assumptions are sufficient to establish well-posedness of the model \eqref{eq:CHC-abstract} and spatio-temporal regularity of the mild solution for $\gamma \in (\tfrac 12,4],d=1$ and $\gamma \in [3,4],d \in \{2,3\}$.
Later, we have extended the results to the case $\gamma \in (\tfrac d2,4],d \in \{1,2,3\}$, which are shown in
\cite[Theorem 3.6]{qi2019sharp}.
The relevant results are presented  in  the following theorem.

\begin{theorem}[Well-posedness and regularity results]
\label{thm:uniqueness-mild-solution}
Under Assumptions \ref{assum:linear-operator-A}--\ref{assum:intial-value-data}, there is a unique mild solution $X: [0,T] \times \Omega \to \dot{H}$ to \eqref{eq:CHC-abstract} given by
\begin{align*}
X(t) = E(t)X_0  - \int_0^t E(t-s) A P F(X(s)) \, \mathrm{d} s
    + \int_0^t E(t-s) \, \mathrm{d} W(s), \, t \in [0, T].
\end{align*}
Furthermore, for $\gamma \in (\tfrac d2,4]$ and $p \geq 1$,
\begin{align*}
\sup_{t\in[0,T]} \|X(t)\|_{L^{p} ( \Omega, \dot{H}^{\gamma} ) } < \infty,
\end{align*}
and for $\alpha \in [0,\gamma]$,
\begin{align*}
\| X(t) - X(s) \|_{L^{p} ( \Omega, \dot{H}^{\alpha} ) }  \leq
 C (t-s)^{\mathrm{min}\{\frac12,\frac{\gamma-\alpha}4\}}, \,  \quad
 0 \leq s \leq t \leq T.
\end{align*}
\end{theorem}

\section{Strong convergence analysis of  numerical approximation}
\label{sec:full}
This section aims to derive strong convergence rates of the numerical discretization, done by a tamed exponential Euler method based on the spectral Galerkin approximation.
\subsection{The spectral Galerkin spatial discretization}
\label{subsec:Galerkin}
We start this part by introducing a finite dimension space spanned by the first $N$ eigenvectors of the dominant linear operator $A$, i.e.,
  $ H_N= {\rm span} \{e_1,\cdots,e_N\} $
and the projection operator $P_N:\dot{H}^{\beta} \to H_N$ is defined by
$P_N x = \sum_{i=1}^N \langle x , e_i \rangle e_i \,\,
\text{for} \,\, \forall x \in \dot{H}^{\beta},~\beta \geq -2$.
Given the identity mapping $I \in \mathcal{L}(\dot{H})$,
one can easily obtain that
 \begin{equation*}
 \big\| \big( P_N - I \big) A^{-\alpha} \big\|_{\mathcal{L}(\dot{H})}
     \leq C \lambda_{N}^{-\alpha}, \quad \forall ~ \alpha \geq 0.
 \end{equation*}
Then applying the spectral Galerkin method to \eqref{eq:CHC-abstract} results in the finite  dimensional stochastic differential equation, given by
\begin{equation}\label{eq:spatial-discrete}
\begin{split}
\left\{
    \begin{array}{lll}
    \dd X^N(t) + A (A X^N(t)  + P_N F( X^N(t) )) \dd t
      = P_N \dd W(t), \quad  t \in (0, T],
    \\
     X^N(0)=P_N X_0,
    \end{array}\right.
\end{split}
\end{equation}
whose unique mild solution is adapted and satisfies
\begin{equation}\label{eq:space-mild}
\begin{split}
X^N(t) & = E(t) P_N X_0
-\int_0^t E(t-s)A P_N F(X^N(s)) \dd s
\\ & \quad + \int_0^t E(t-s) P_N \dd W(s).
\end{split}
\end{equation}
The following theorem, concerning the strong convergence rate of the spectral Galerkin method, is an immediate consequence of
\cite[Theorem 3.5]{qi2019sharp}.
\begin{theorem}\label{th:space-rate}
Let $X(t)$ be the mild solution of \eqref{eq:CHC-abstract} and
let $X^N(t)$ be the solution of \eqref{eq:spatial-discrete}.
Suppose Assumptions \ref{assum:linear-operator-A}--\ref{assum:intial-value-data}
are valid, then for any $p \in [1,\infty)$, it holds that
\begin{equation*}
\sup_{t \in [0,T]}
\| X(t) - X^N(t) \|_{L^p(\Omega,\dot{H})} \leq C \lambda_N^{-\frac \gamma2}.
\end{equation*}
\end{theorem}

\subsection{An explicit fully  discrete scheme and its a priori moment bounds}
\label{subsec:moment bounds}
This subsection concerns the a priori moment bounds of a spatio-temporal full discretization based on the spatial spectral Galerkin approximation.
In order to introduce the fully discrete scheme, we define the nodes $t_m=m\tau$ with a uniform time step-size $\tau=\tfrac{T}{M}$ for
$m \in \{0, 1,\ldots,M \},\,  M \in \mathbb{N}^{+}$ and introduce a notation
$\lfloor {t} \rfloor_{\tau} :=t_i$ for
$ t \in [t_i ,t_{i+1}), i \in \{0,1,\ldots , M-1\}$.
It is worthwhile to mention that the fully discrete exponential Euler and fully discrete linear-implicit Euler approximations diverge strongly and numerically weakly in the case of stochastic Allen--Cahn equations \cite{beccari2019strong}.
Thus, we apply the tamed exponential Euler scheme to \eqref{eq:spatial-discrete} and get
\begin{equation}\label{full;discrete version}
\begin{split}
X_{t_{m+1}}^{M,N} & = E (\tau)X_{t_m}^{M,N}
     - \tfrac{ A^{-1} (I-E(\tau)) P_N F (X_{t_m}^{M,N})}
        {1+\tau \| P_N F (X_{t_m}^{M,N}) \|}
  \\ & \quad   + \int_{t_m}^{t_{m+1}} E(t_{m+1}-\lfloor {s} \rfloor_{\tau})
         P_N  \dd W(s).
\end{split}
\end{equation}
Particularly, the following continuous version of
 \eqref{full;discrete version} will be used frequently,
\begin{equation}\label{full;continuous version}
X_{t}^{M,N} = E (t) P_N X_0
   - \int_0^t \tfrac{ E (t-s) A P_N F (X_{\lfloor {s} \rfloor_{\tau}}^{M,N})}
       { 1 + \tau \| P_N F (X_{\lfloor {s} \rfloor_{\tau}}^{M,N}) \| } \dd s
   + \mathcal{O}_t^{M,N},
\end{equation}
which is $\mathcal{F}_t$-adapted.
Here for simplicity of presentation we denote
\begin{equation*}
\mathcal{O}_t^{M,N} := \int_0^t E( t- \lfloor {s} \rfloor_{\tau}) P_N \dd W(s),
\end{equation*}
which satisfies the following regularity result.

\begin{lemma}\label{lem:discrete-stochastic-convolution}
Suppose Assumptions \ref{assum:linear-operator-A} and \ref{assum:noise-term} hold.
Then for all $ p \ge 1$ and
$\theta \in  [ 0, \min \{ \gamma, 2 \})$,
the discrete stochastic convolution $\mathcal{O}_t^{M,N}$ satisfies
\begin{equation*}
\E \Big[ \sup_{t\in[0,T]} \vert \mathcal{O}_t^{M,N} \vert_{\theta}^p \Big] < \infty.
\end{equation*}
\end{lemma}
\begin{proof}
Following a similar approach used in \cite[Theorem 5.10]{prato2014stochastic},
we can rewrite $\mathcal{O}_t^{M,N}$ as
\begin{equation*}
\mathcal{O}_t^{M,N} = \frac{\sin(\alpha \pi)}{\pi}
  \int_0^t ( t - s )^{\alpha - 1} E( t - s ) Y_{\alpha} (s) \dd s,
  \,\, \alpha \in (0,1)
\end{equation*}
with
\begin{equation*}
 Y_{\alpha} (s) := \int_0^s ( s - r )^{-\alpha}
                   E (s - \lfloor r \rfloor) P_N \dd W(r).
\end{equation*}
Indeed, by stochastic Fubini theorem, we get
\begin{equation*}
\begin{split}
\frac{\sin(\alpha \pi)}{\pi}
 & \int_0^t ( t - s )^{\alpha - 1} E( t - s ) Y_{\alpha} (s) \dd s
 \\ & = \frac{\sin(\alpha \pi)}{\pi}
  \int_0^t ( t - s )^{\alpha - 1} E( t - s )
  \Big[ \int_0^s ( s - r )^{-\alpha}
                   E (s - \lfloor r \rfloor) P_N \dd W(r)\Big] \dd s
 \\ & = \frac{\sin(\alpha \pi)}{\pi}
  \int_0^t \Big[ \int_r^t ( t - s )^{\alpha - 1} ( s - r )^{-\alpha}
        \dd s \Big] E ( t - \lfloor r \rfloor ) P_N \dd W(r)
 \\ & = \int_0^t E( t- \lfloor {r} \rfloor_{\tau}) P_N \dd W(r),
\end{split}
\end{equation*}
where a basic fact
\begin{equation*}
\int_r^t ( t - s )^{\alpha - 1} ( s - r )^{-\alpha} \dd s
 =  \frac{\pi}{\sin(\alpha \pi)},
 \, 0 \leq r \leq t, \, \alpha \in (0,1)
\end{equation*}
was invoked in the last equality.
As a result, using the Burkholder--Davis--Gundy inequality and H\"older's inequality leads to
\begin{equation*}
\begin{split}
\E \Big[  \sup_{t\in[0,T]} \vert \mathcal{O}_t^{M,N} \vert_{\theta}^p \Big]
 &= \E \Big[ \sup_{t\in[0,T]} \Big\vert \frac{\sin(\alpha \pi)}{\pi}
  \int_0^t ( t - s )^{\alpha - 1} E( t - s ) Y_{\alpha} (s) \dd s \Big\vert_{\theta}^p \Big]
  \\ & \leq C \, \E \Big[ \sup_{t\in[0,T]} \Big(
  \int_0^t ( t - s )^{\alpha - 1 -\frac \theta 4} \vert Y_{\alpha} (s) \vert \dd s \Big)^p \Big]
  \\ & \leq C \, \E \Big[ \sup_{t\in[0,T]} \Big(
  \int_0^t ( t - s )^{(\alpha - 1 -\frac \theta 4)q} \dd s \Big)^\frac pq
  \cdot \int_0^t \vert Y_{\alpha} (s) \vert^p \dd s  \Big]
  \\ & \leq C \int_0^T \E \big[ \vert Y_{\alpha} (s) \vert^p \big] \dd s
   \\ & \leq C \int_0^T \Big( \int_0^s ( s - r )^{ - 2 \alpha }
        \| E ( s- \lfloor {r} \rfloor_{\tau}) Q^{\frac 12}
            \|_{\mathcal{L}_2}^2 \dd r \Big)^{\frac p2} \dd s
   \\ & \leq C \int_0^T \Big( \int_0^s
          ( s - r )^{ - 2 \alpha + \min \{ \frac {\gamma-2}2,0\} }
        \| A^{\frac {\gamma-2}2} Q^{\frac 12}
            \|_{\mathcal{L}_2}^2 \dd r \Big)^{\frac p2} \dd s
   \\ & < \infty,
\end{split}
\end{equation*}
where $\alpha > \tfrac 1p + \tfrac \theta4 $ was used in the third inequality and $\alpha \in (0, \min \{ \tfrac \gamma 4, \tfrac 12 \})$ was used in the last inequality.
Finally, choosing sufficiently large $p>1$ and $\theta < \min \{ \gamma, 2\}$ completes the proof.
\end{proof}

The forthcoming lemma is a direct consequence of \cite[Lemma 3.2]{qi2019sharp}, which is crucial to the moment bound and convergence analysis.

\begin{lemma}\label{lem:H-1}
Let $F:L^6 \rightarrow H$ be the Nemytskii operator in Assumption \ref{assum:nonlinearity}.
Then it holds for any  $\iota \in (\tfrac12,1)$ and $d=1$,
\begin{equation*}
\vert F'(u)v \vert_{\iota} \leq C \big( 1 + \vert u \vert_{\iota}^2 \big) \vert v \vert_{1},
\,\, u \in \dot{H}^{\iota}, \, v \in \dot{H}^1,
\end{equation*}
and for any $\iota \in (\tfrac d2,2)$, $d=2,3$,
\begin{equation*}
\vert F'(u)v \vert_1 \leq C \big( 1 + \vert u \vert_{\iota}^2 \big) \vert v \vert_{1},
\,\, u \in \dot{H}^{\iota}, \, v \in \dot{H}^1.
\end{equation*}
\end{lemma}

Next we construct a sequence of decreasing subevents
\begin{equation*}
\Omega_{R,t_i}= \big\{ \omega \in \Omega :\sup\limits_{j\in\{0,1,\ldots,i\}}
   \| X_{t_j}^{M,N}(\omega) \|_{L^6} \! \leq R \big\},
      R \in (0,\infty), i \in \{0,1,\ldots,M\}.
\end{equation*}
By $\Omega^c$ and $\chi_{\Omega}$ we denote the complement and indicator function of a set $\Omega$, respectively.
It is easy to see that $ \chi_{\Omega_{R,t_i}}$ is
$ \mathcal{F}_{t_i}$-adapted
and $\chi_{\Omega_{R,t_i}} \leq \chi_{\Omega_{R,t_j}}$ for $t_i \geq t_j$.
Besides, we introduce a process $Y_t^{M,N}$ by
\begin{equation*}
Y_t^{M,N} := X_t^{M,N} - \mathcal{O}_t^{M,N}
   = E(t) P_N X_0 - \int_0^t
    \tfrac{ E(t-s) A P_N F (X_{\lfloor {s} \rfloor_{\tau}}^{M,N})}
          { 1 + \tau \| P_N F (X_{\lfloor {s} \rfloor_{\tau}}^{M,N}) \|}\dd s,
\end{equation*}
which can be rewritten as
\begin{equation*}
Y_t^{M,N} = E(t) P_N X_0 -
\int_0^t \!\! E(t-s) A P_N F( X_s^{M,N} ) \dd s
     + \int_0^t \!\! E(t-s) A P_N Z_s^{M,N} \dd s,
\end{equation*}
with
$Z_t^{M,N} := F(X_t^{M,N}) -
    \tfrac{F(X_{\lfloor {t} \rfloor_{\tau}}^{M,N})}
   {1+ \tau \| P_N F(X_{\lfloor {t} \rfloor_{\tau}}^{M,N}) \|}$.
Then, for $t \in (0,T]$, $Y_t^{M,N}$ satisfies
\begin{equation}\label{eq:moment-Yt}
\tfrac{\dd}{\dd t} Y_t^{M,N} + A^2 Y_t^{M,N}
    + A P_N F( Y_t^{M,N} + \mathcal{O}_t^{M,N} ) = A P_N Z_t^{M,N}.
\end{equation}
Equipped with the above preparations, we are ready to present the following two lemmas, which aim to bound the numerical approximations on the well-chosen subevents $\Omega_{R_{\tau},t_{i-1}}$.
\begin{lemma}\label{lem:Ys-bound-subevents1}
Suppose Assumptions \ref{assum:linear-operator-A}--\ref{assum:intial-value-data} are valid.
Let $p \in [1,\infty)$ and $R_\tau = \tau^{- \min \{  \frac 4{81}, \frac {\gamma}{24}\}}$ for $\gamma \in (\tfrac d2,4]$ coming from \eqref{eq:ass-AQ-condition}.
Then for $i \in \{0,1,\ldots,M\}$,
\begin{align}\label{eq:moment-1-A}
\begin{split}
\Big\| \sup_{s\in[0,t_i]} \chi_{\Omega_{R_{\tau},t_{i-1}}}
 & Y_{s}^{M,N} \Big\|_{L^{2p}(\Omega,\dot{H})}^2
    + \Big\|  \int_0^{t_i} \chi_{\Omega_{R_{\tau},t_{i-1}}}
      \| A Y_s^{M,N} \|^2 \dd s \Big\|_{L^p(\Omega,\mathbb{R})}
  \\ & + \Big\| \int_0^{t_i} \chi_{\Omega_{R_{\tau},t_{i-1}}}
        \| \nabla [(Y_s^{M,N})^2] \|^2 \dd s
      \Big\|_{L^p(\Omega,\mathbb{R})} <\infty.
\end{split}
\end{align}
\end{lemma}
\begin{proof}
%Following the procedure from the deterministic case,
One can easily derive from \eqref{eq:moment-Yt} that
%Applying $\langle \cdot, A^{-1} Y^{M,N}_t \rangle$ to \eqref{eq:moment-Yt}, we obtain
\begin{equation*}
\begin{split}
\langle \tfrac{\dd}{\dd t} & Y_t^{M,N} , A^{-1} Y_t^{M,N} \rangle
 + \langle A^2 Y_t^{M,N}, A^{-1} Y_t^{M,N} \rangle
 \\& \quad  + \langle A P_N F( Y_t^{M,N} + \mathcal{O}_t^{M,N} ) , A^{-1} Y_t^{M,N} \rangle
     = \langle A P_N Z_t^{M,N} , A^{-1} Y_t^{M,N} \rangle,
\end{split}
\end{equation*}
which can be rewritten as
\begin{equation*}
\tfrac12 \tfrac{\dd}{\dd t} \vert Y_t^{M,N}\vert_{-1}^2
   + \vert Y_t^{M,N} \vert_1^2
   + \langle  F( Y_t^{M,N} + \mathcal{O}_t^{M,N} ),  Y_t^{M,N} \rangle
   = \langle  Z_t^{M,N},  Y_t^{M,N} \rangle.
\end{equation*}
Integrating over $[0,t_i]$ and then using Young's inequality yield
\begin{equation*}
\begin{split}
\vert & Y_{t_i}^{M,N}  \vert_{-1}^2 - \vert Y_0^{M,N} \vert_{-1}^2
   \\ & = - 2 \int_0^{t_i} \vert Y_s^{M,N} \vert_1^2 \dd s
          - 2 \int_0^{t_i}  \big\langle F(Y_s^{M,N} + \mathcal{O}_s^{M,N}),
            Y_s^{M,N} \big\rangle \dd s
   \\ & \quad  + 2 \int_0^{t_i} \big\langle Z_s^{M,N}, Y_s^{M,N} \big\rangle \dd s
   \\ & = - 2 \int_0^{t_i}  \vert Y_s^{M,N} \vert_1^2 \dd s
          - 2 \int_0^{t_i}  \|Y_s^{M,N}\|_{L^4}^4 \dd s
          + 2 \int_0^{t_i}  \|Y_s^{M,N}\|^2 \dd s
   \\ & \quad - 2 \int_0^{t_i} \big\langle 3(Y_s^{M,N})^2 \mathcal{O}_s^{M,N}
          + 3 Y_s^{M,N} (\mathcal{O}_s^{M,N})^2 + (\mathcal{O}_s^{M,N})^3
          - \mathcal{O}_s^{M,N}, Y_s^{M,N} \big\rangle \dd s
   \\ & \quad   + 2 \int_0^{t_i}  \big\langle Z_s^{M,N}, Y_s^{M,N}
   \big\rangle \dd s
   \\ & \leq  -  \int_0^{t_i} \vert Y_s^{M,N} \vert_1^2 \dd s
              -  \int_0^{t_i} \|Y_s^{M,N}\|_{L^4}^4 \dd s
              + C \int_0^{t_i} \vert Y_s^{M,N} \vert_{-1}^2 \dd s
   \\ & \quad + C \int_0^{t_i} \vert Z_s^{M,N} \vert_{-1}^2 \dd s
    + C \int_0^{t_i} \big( 1 + \|\mathcal{O}_s^{M,N}\|_{L^4}^4 \big) \dd s.
\end{split}
\end{equation*}
It follows from Gronwall's inequality that
\begin{equation*}
\vert Y_{t_i}^{M,N} \vert_{-1}^2 \leq C \Big( \vert Y_0^{M,N} \vert_{-1}^2
     + \int_0^{t_i} \vert Z_s^{M,N} \vert_{-1}^2 \dd s
     + \int_0^{t_i} ( 1 + \| \mathcal{O}_s^{M,N} \|_{L^4}^4 ) \dd s \Big),
\end{equation*}
which implies that
\begin{equation}\label{eq;v_n1}
\begin{split}
\int_0^{t_i} & \vert Y_s^{M,N} \vert_{1}^2 \dd s
   + \int_0^{t_i}  \|Y_s^{M,N}\|_{L^4}^4 \dd s
     \\ & \leq C \Big( \vert Y_0^{M,N} \vert_{-1}^2
      + \int_0^{t_i}  \vert  Z_s^{M,N}  \vert_{-1}^2 \dd s
      + \int_0^{t_i}  ( 1 + \|\mathcal{O}_s^{M,N}\|_{L^4}^4 ) \dd s \Big).
\end{split}
\end{equation}
Since $\|Y_0^{M,N}\|_{L^p(\Omega,\dot{H}^{\gamma})} < \infty$ and
 $\|\mathcal{O}_t^{M,N}\|_{L^p(\Omega,\dot{H}^{\gamma})} < \infty$, we deduce
\begin{equation*}
\begin{split}
\Big\| \int_0^{t_i} \!\! \vert Y_s^{M,N} \vert_{1}^2 \dd s \Big\|_{L^p(\Omega,\mathbb{R})}
  & + \Big\| \int_0^{t_i} \!\! \|Y_s^{M,N}\|_{L^4}^4 \dd s
         \Big\|_{L^p(\Omega,\mathbb{R})}
  \\ &  \leq C \Big( 1 + \int_0^{t_i} \| Z_s^{M,N}
        \|_{L^{2p}(\Omega,\dot{H}^{-1})}^2 \dd s \Big).
\end{split}
\end{equation*}
Taking inner product of \eqref{eq:moment-Yt} by $Y_t^{M,N}$ and integrating from $0$ to $t_i$ lead to
\begin{footnotesize}
\begin{equation*}
\begin{split}
 \| & Y_{t_i }^{M,N} \|^2  - \| Y_0^{M,N} \|^2
 \\ & = - 2 \int_0^{t_i}  \|A Y_s^{M,N}\|^2 \dd s
        - 2 \int_0^{t_i}  \langle A F
        \big( Y_s^{M,N} + \mathcal{O}_s^{M,N} \big), Y_s^{M,N} \rangle \dd s
        + 2 \int_0^{t_i}\langle Z_s^{M,N}, A Y_s^{M,N} \rangle \dd s
 \\ & \leq - \int_0^{t_i}  \| A Y_s^{M,N} \|^2 \dd s
        - 6 \int_0^{t_i} \| Y_s^{M,N} \nabla Y_s^{M,N} \|^2 \dd s
        + 2 \int_0^{t_i}  \vert Y_s^{M,N} \vert_{1}^2 \dd s
        +   \int_0^{t_i}  \|Z_s^{M,N}\|^2 \dd s
 \\ & \quad - 2 \int_0^{t_i} \langle 3 ( Y_s^{M,N})^2  \mathcal{O}_s^{M,N}
        + 3 Y_s^{M,N} ( \mathcal{O}_s^{M,N} )^2
        + (\mathcal{O}_s^{M,N})^3 - \mathcal{O}_s^{M,N},
          A Y_s^{M,N}  \rangle \dd s
 \\ & \leq - \tfrac 12 \int_0^{t_i}  \|A Y_s^{M,N}\|^2 \dd s
           - \tfrac32  \int_0^{t_i}  \|\nabla [( Y_s^{M,N} )^2] \|^2 \dd s
           + 2 \int_0^{t_i} \vert Y_s^{M,N} \vert_1^2 \dd s
           + \int_0^{t_i}  \|Z_s^{M,N}\|^2 \dd s
 \\ & \quad + C \int_0^{t_i} \Big(
               \|Y_s^{M,N}\|_{L^4}^4 \|\mathcal{O}_s^{M,N}\|_V^2
            +  \|Y_s^{M,N}\|^2  \|\mathcal{O}_s^{M,N}\|_V^4
            +  \|\mathcal{O}_s^{M,N}\|_{L^6}^6
            +  \|\mathcal{O}_s^{M,N}\|^2 \Big) \dd s.
\end{split}
\end{equation*}
\end{footnotesize}
Again, the use of \eqref{eq;v_n1} gives
\begin{equation*}
\begin{split}
\| Y_{t_i}^{M,N} \|^2 & + \int_0^{t_i} \| A Y_s^{M,N} \|^2 \dd s
   +  \int_0^{t_i}  \|\nabla [ (Y_s^{M,N})^2 ] \|^2 \dd s
    \\ & \leq C \Big( 1+\sup_{ s \in [0,T]} \|\mathcal{O}_s^{M,N}\|_V^8 \Big)
     \big( 1 + \| Y_0^{M,N} \|^2 + \int_0^{t_i} \| Z_s^{M,N} \|^2 \dd s \big).
\end{split}
\end{equation*}
At the moment, we turn to the estimate
$ \E \big[ \chi_{\Omega_{R_{\tau},t_{i-1}}} \|Z_s^{M,N}\|^{2p} \big] $.
For $s \in [0,t_i]$, we have
\begin{equation*}
\begin{split}
 & \chi_{\Omega_{R_{\tau},t_{i-1}}}  \| Z_s^{M,N} \|
 \\  & \quad \leq
\chi_{\Omega_{R_{\tau},t_{i-1}}} \big\| F (X_s^{M,N})
    - F ( X_{\lfloor {s} \rfloor_{\tau}}^{M,N} ) \big\|
 \\  & \qquad + \chi_{\Omega_{R_{\tau},t_{i-1}}}
  \Big\| F (X_{\lfloor {s} \rfloor_{\tau}}^{M,N})
    - \tfrac{ F (X_{\lfloor {s} \rfloor_{\tau}}^{M,N}) }
    { 1 + \tau \| P_N F (X_{\lfloor {s} \rfloor_{\tau}}^{M,N}) \|} \Big\|
 \\ & \quad  \leq C \, \chi_{\Omega_{R_{\tau},t_{i-1}}}
  \big( 1 + \| X_{\lfloor {s} \rfloor_{\tau}}^{M,N} \|_{V}^2
         +  \| X_s^{M,N} \|_{V}^2 \big)
  \big( \| X_s^{M,N} - X_{\lfloor {s} \rfloor_{\tau}}^{M,N} \| \big)
 \\ & \qquad + C \, \chi_{\Omega_{R_{\tau},t_{i-1}}}
     \tau \| F (X_{\lfloor {s} \rfloor_{\tau}}^{M,N}) \|^2.
\end{split}
\end{equation*}
Before further proof, we claim that
\begin{equation}\label{estimate-XSMN}
\chi_{ \Omega_{R_{\tau},t_{i-1}} } \| X_s^{M,N} \|_{V}
  \leq C \big( 1 + \|X_0\|_{V} + R_{\tau}^3 +
       \| \mathcal{O}_s^{M,N} \|_{V} \big),
       \, \forall \, s \in [0,t_i).
\end{equation}
Indeed, by stability of the semigroup $E(t)$ in $V$ and Sobolev embedding inequality $\dot{H}^{\delta} \subset V, \delta > \tfrac d2$,
\begin{equation*}
\begin{split}
& \chi_{ \Omega_{R_{\tau},t_{i-1}} } \| X_s^{M,N} \|_{V}
  \\ & \quad \leq \chi_{ \Omega_{R_{\tau},t_{i-1}} }  \Big(
  \| E(s) P_N X_0\|_{V} + \int_0^s \| E ( s - r ) A
           F ( X_{\lfloor {r} \rfloor_{\tau}}^{M,N}) \|_{V} \dd r
   + \| \mathcal{O}_s^{M,N}\|_{V} \Big)
  \\ & \quad \leq \chi_{ \Omega_{R_{\tau},t_{i-1}} }  \Big(
        \| X_0\|_{V} + \int_0^s  ( s - r )^{- \frac {2+\delta}4}
          \| P F ( X_{\lfloor {r} \rfloor_{\tau}}^{M,N}) \| \dd r
           + \| \mathcal{O}_s^{M,N} \|_{V}  \Big)
  \\ & \quad \leq C \big( 1 + \|X_0\|_{V} + R_{\tau}^3 +
       \| \mathcal{O}_s^{M,N} \|_{V} \big).
\end{split}
\end{equation*}
Next, owing to  \eqref{full;continuous version}, one can write
\begin{equation*}
\begin{split}
X_s^{M,N} - X_{\lfloor {s} \rfloor_{\tau}}^{M,N}
    &  = [ E (s) - E (\lfloor {s} \rfloor_{\tau}) ] P_N X_0
   - \int_0^s \tfrac { E(s-u)A P_N F (X_{\lfloor {u} \rfloor_{\tau}}^{M,N})}
         { 1 + \tau \| P_N F (X_{\lfloor {u} \rfloor_{\tau}}^{M,N}) \| } \dd u
 \\ & \quad + \int_0^{\lfloor {s} \rfloor_{\tau}}
    \tfrac { E( \lfloor {s} \rfloor_{\tau} - u )A P_N
       F (X_{\lfloor {u} \rfloor_{\tau}}^{M,N}) }
    { 1 + \tau \|P_N F (X_{\lfloor {u} \rfloor_{\tau}}^{M,N}) \|} \dd u
   + \mathcal{O}_s^{M,N} - \mathcal{O}_{\lfloor {s} \rfloor_{\tau}}^{M,N},
\end{split}
\end{equation*}
which implies that
\begin{equation*}
\begin{split}
 & \chi_{ \Omega_{R_{\tau},t_{i-1}} }
    \| X_s^{M,N} - X_{\lfloor {s} \rfloor_{\tau}}^{M,N} \|
 \\ &  \leq \tau^{ \frac{\gamma}{4} } \vert X_0 \vert_{\gamma}
   + \chi_{\Omega_{R_{\tau},t_{i-1}}} \Big\|
       \int_0^{\lfloor {s} \rfloor_{\tau}} \!\!
        E( \lfloor {s} \rfloor_{\tau} - u )
       ( E( s - \lfloor {s} \rfloor_{\tau} ) - I )
       \tfrac { A P_N F (X_{\lfloor {u} \rfloor_{\tau}}^{M,N}) }
         { 1 + \tau \| P_N F (X_{\lfloor {u} \rfloor_{\tau}}^{M,N}) \| } \dd u
            \Big\|
\\ & \quad + \chi_{\Omega_{R_{\tau},t_{i-1}}}
       \int_{\lfloor {s} \rfloor_{\tau}}^s \Big\|
     \tfrac { E(s- u) A P_N F (X_{\lfloor {u} \rfloor_{\tau}}^{M,N}) }
        { 1 + \tau \|P_N F (X_{\lfloor {u} \rfloor_{\tau}}^{M,N}) \|}
         \Big\|  \dd u
   + \chi_{\Omega_{R_{\tau},t_{i-1}}} \| \mathcal{O}_s^{M,N}
      - \mathcal{O}_{\lfloor {s} \rfloor_{\tau}}^{M,N} \|
\\ &  \leq \tau^{\frac{\gamma}{4}} \vert X_0 \vert_{\gamma}
    + C \big( 1 + R_{\tau}^3 \big )
    \Big( \tau^{\frac 49}
      \int_0^{\lfloor {s} \rfloor_{\tau}}
     ( \lfloor {s} \rfloor_{\tau} - u )^{-\frac {17}{18}} \dd u
     + \int_{\lfloor {s} \rfloor_{\tau}}^s ( s - u)^{-\frac 12} \dd u
     \Big)
\\ & \quad   + \|\mathcal{O}_s^{M,N} -
             \mathcal{O}_{\lfloor {s} \rfloor_{\tau}}^{M,N}\|
\\ &  \leq \tau^{\frac{\gamma}{4}} \vert X_0 \vert_{\gamma}
    + C \big( 1 + R_{\tau}^3 \big )
     \big( \tau^{\frac 49} + \tau^{\frac 12} \big)
    + \|\mathcal{O}_s^{M,N} -
             \mathcal{O}_{\lfloor {s} \rfloor_{\tau}}^{M,N}\|
\\ &  \leq \tau^{\frac{\gamma}{4}} \vert X_0 \vert_{\gamma}
    + C \big( 1 + R_{\tau}^3 \big )
      \tau^{\frac 49}
    + \|\mathcal{O}_s^{M,N} -
             \mathcal{O}_{\lfloor {s} \rfloor_{\tau}}^{M,N}\|.
\end{split}
\end{equation*}
Therefore,
\begin{equation*}
\begin{split}
\chi_{\Omega_{R_{\tau},t_{i-1}}}    \| Z_s^{M,N}\|
   & \leq C \, \tau ( 1 + R_{\tau}^6)
    \\ & \quad +  C \big( 1 + \| X_0 \|_{V}^2 +  R_{\tau}^6
         + \| \mathcal{O}_s^{M,N} \|_{V}^2
         + \| \mathcal{O}_{\lfloor {s} \rfloor_{\tau}}^{M,N} \|_{V}^2 \big)
   \\ & \quad   \times \Big(
       \tau^{ \frac { \gamma } {4} } \vert X_0 \vert_{\gamma}
           + \big( 1 + R_{\tau}^3 \big )
            \tau^{\frac 49}
           + \|\mathcal{O}_s^{M,N} -
           \mathcal{O}_{\lfloor {s} \rfloor_{\tau}}^{M,N} \| \Big).
\end{split}
\end{equation*}
Note that
 \begin{equation*}
 \E \Big[  \big\| \mathcal{O}_s^{M,N} -
    \mathcal{O}_{\lfloor {s} \rfloor_{\tau}}^{M,N} \big\|^p \Big]
    \leq C \tau^{ \mathrm{min} \{ \frac12,\frac{\gamma}{4} \}p }.
 \end{equation*}
Then taking $R_{\tau}=\tau^{- \min \{ \frac 4{81}, \frac {\gamma}{24} \}}$ leads to
 \begin{equation*}
\E \Big[ \chi_{\Omega_{R_{\tau},t_{i-1}}} \| Z_s^{M,N} \|^{2p} \Big] < \infty.
 \end{equation*}
As a result, we can deduce  that
\begin{align*}
\begin{split}
\Big\| \sup_{s\in[0,t_i]} \chi_{\Omega_{R_{\tau},t_{i-1}}}
 & Y_{s}^{M,N} \Big\|_{L^{2p}(\Omega,\dot{H})}^2
    + \Big\|  \int_0^{t_i} \chi_{\Omega_{R_{\tau},t_{i-1}}}
      \| A Y_s^{M,N} \|^2 \dd s \Big\|_{L^p(\Omega,\mathbb{R})}
  \\ & + \Big\| \int_0^{t_i} \chi_{\Omega_{R_{\tau},t_{i-1}}}
        \| \nabla [(Y_s^{M,N})^2] \|^2 \dd s
      \Big\|_{L^p(\Omega,\mathbb{R})} <\infty.
\end{split}
\end{align*}
This completes the proof.
\end{proof}

With the aid of Lemma \ref{lem:Ys-bound-subevents1},
we are able to obtain $p$-th moment bounds of
$\| X_{t_i}^{M,N} \|_{L^6}$ on the subevents
$\Omega_{R_{\tau},t_{i-1}}$.

\begin{lemma}\label{indication;a priori}
Let $p \in [1,\infty)$,
$R_\tau = \tau^{- \min \{  \frac 4{81}, \frac {\gamma}{24}\}}$ for $d=1$ and $R_\tau = \tau^{- \min \{ \frac 1{54}, \frac {\gamma-1}{36}\}}$ for $d=2,3$.
Under Assumptions \ref{assum:linear-operator-A}--\ref{assum:intial-value-data},
the fully discrete solution $X_{t_i}^{M,N}$ satisfies
\begin{equation*}
\sup_{M,N \in \N^{+}} \sup_{i \in \{0,1,\ldots,M\}}
 \E \big[ \chi_{\Omega_{R_{\tau},t_{i-1}}}
    \| X_{t_i}^{M,N} \|_{L^6}^p \big] < \infty,
\end{equation*}
where with the convention, we set $\chi_{\Omega_{R_{\tau},t_{-1}}}=1$.
\end{lemma}

\begin{proof}
Combining \eqref{eq:moment-1-A} with  Sobolev embedding inequality and Gagliardo--Nirenberg inequality implies that for $d=1$,
\begin{equation}\label{eq:moment-L6-1d}
\begin{split}
\E &\Big[\chi_{\Omega_{R_{\tau},t_{i-1}}} \| Y_{t_i}^{M,N} \|_{L^6}^p \Big]
   \\ & \leq \E \Big[ \chi_{\Omega_{R_{\tau},t_{i-1}}}
       \|E(t_i)  X_0\|_{L^6}^p \Big]
 \\ & \qquad   + \E \Big[ \Big\| \int_0^{t_i} \!\! \chi_{\Omega_{R_{\tau},t_{i-1}}}
         E( t_i - s ) A P_N ( Z_s^{M,N} - F(X_s^{M,N}) )\dd s
               \Big\|_{L^6}^p \Big]
  \\ & \leq C + C \E \Big[ \int_0^{t_i} ( t_i - s )^{-\frac{7}{12}}
  \\ & \qquad  \times   \big( 1 + \|\mathcal{O}_s^{M,N}\|_{L^6}^3
      + \chi_{\Omega_{R_{\tau},t_{i-1}}} \| Y_{s}^{M,N} \|_{L^6}^3
      + \chi_{\Omega_{R_{\tau},t_{i-1}}} \| Z_s^{M,N} \|
     \big) \dd s \Big]^p
  \\ & \leq C + \E \Big[ \int_0^{t_i} ( t_i - s )^{-\frac{7}{12}}
    \big( \chi_{\Omega_{R_{\tau},t_{i-1}}} \| A Y_{s}^{M,N} \|^{\frac12}
       \| Y_{s}^{M,N} \|^{\frac52} \big) \dd s \Big]^p
  \\ & \leq C + C \Big( \int_0^{t_i} ( t_i - s )^{-\frac{7}{9}} \dd s
       \Big)^{p}  \E \big[ \sup_{ s \in [0,t_i] }
        \chi_{\Omega_{R_{\tau},t_{i-1}}} \|Y_{s}^{M,N}\|^{\frac {10p}3} \big]
  \\ & \quad  + C \E \Big[ \int_0^{t_i} \chi_{\Omega_{R_{\tau},t_{i-1}}}
         \| A Y_{s}^{M,N} \|^2 \dd s \Big]^{p}
  \\ & < \infty.
\end{split}
\end{equation}
For  $d=2,3$, we define a Lyapunov functional $J(u)$ by
\begin{align*}
J(u)= \frac12 \|\nabla u\|^2 + \int_{\mathcal{D}} \Phi(u) \dd x,
\end{align*}
where $\Phi$ is the primitive of $F$.
Multiplying \eqref{eq:moment-Yt} by $A^{-1} \dot{Y}_t^{M,N}$ yields
\begin{equation*}
\big \vert \dot{Y}_t^{M,N} \big \vert_{-1}^2
  + \tfrac12 \tfrac{ \dd \vert Y_t^{M,N} \vert_1^2 } { \dd t }
  + \langle F( Y_t^{M,N} + \mathcal{O}_t^{M,N} ), \dot{Y}_t^{M,N} \rangle
  = \langle Z_t^{M,N}, \dot{Y}_t^{M,N} \rangle.
\end{equation*}
To proceed further, owing to H\"{o}lder's inequality,
for $\theta > \tfrac d2$, we have
\begin{equation*}
\begin{split}
\| ( Y_t^{M,N} )^2 \nabla \mathcal{O}_t^{M,N} \|
  & \leq \| ( Y_t^{M,N} )^2 \|_{L^{ \frac{2(2+\delta)}{\delta} }}
        \| \nabla \mathcal{O}_t^{M,N} \|_{L^{2+\delta}}
  \\ & \leq C \vert ( Y_t^{M,N}  )^2 \vert_1 \vert \mathcal{O}_t^{M,N} \vert_{\theta},
\end{split}
\end{equation*}
where the Sobolev embedding inequality
$ \dot{H}^{\frac d2 - \frac dp} \subset L^p$ for $p \ge 2$ was used in the last inequality and we take sufficiently small $\delta > 0$ for $d=2$ and $\delta = 1$ for $d=3$.
Therefore,
\begin{equation*}
\begin{split}
\vert P \big( & 3 (Y_t^{M,N})^2 \mathcal{O}_t^{M,N}
        + 3 Y_t^{M,N} (\mathcal{O}_t^{M,N})^2 + (\mathcal{O}_t^{M,N})^3
        - \mathcal{O}_t^{M,N} \big) \vert_1
  \\ & \leq C \big( \vert (Y_t^{M,N})^2 \vert_1 \| \mathcal{O}_t^{M,N} \|_V
         + \| ( Y_t^{M,N} )^2 \nabla \mathcal{O}_t^{M,N} \|
         + \| \nabla Y_t^{M,N} \| \| \mathcal{O}_t^{M,N} \|_V^2
  \\ & \quad + \| Y_t^{M,N} \|_V \| \mathcal{O}_t^{M,N} \|_V
         \| \nabla \mathcal{O}_t^{M,N} \|
         +  \| \nabla \mathcal{O}_t^{M,N} \|
              \| \mathcal{O}_t^{M,N} \|_V^2
         +  \| \nabla \mathcal{O}_t^{M,N} \| \big)
  \\ & \leq C \big( 1 + \vert (Y_t^{M,N})^2 \vert_1 + \| A  Y_t^{M,N} \| \big)
         \big( 1 + \vert \mathcal{O}_t^{M,N} \vert_{\theta}^3 \big).
\end{split}
\end{equation*}
Combining it with the fact $\Phi'(t)=F(t)$ and Cauchy--Schwartz inequality infers that for $\theta > \tfrac d2$,
\begin{align*}
\begin{split}
- & \langle F( Y_t^{M,N} + \mathcal{O}_t^{M,N}), \dot{Y}_t^{M,N} \rangle
    + \langle Z_t^{M,N}, \dot{Y}_t^{M,N} \rangle
  \\ & = - \langle  F( Y_t^{M,N} ), \dot{Y}_t^{M,N} \rangle
         + \langle  Z_t^{M,N}, \dot{Y}_t^{M,N} \rangle
  \\ & \quad - \langle 3 ( Y_t^{M,N} )^2 \mathcal{O}_t^{M,N}
        + 3 Y_t^{M,N} (\mathcal{O}_t^{M,N})^2 + (\mathcal{O}_t^{M,N})^3
          - \mathcal{O}_t^{M,N},\dot{Y}_t^{M,N} \rangle
  \\ & \leq - \tfrac{\dd }{\dd t} \int_{\mathcal{D}} \Phi(Y_t^{M,N}) \dd x
            + \vert Z_t^{M,N} \vert_1 \vert \dot{Y}_t^{M,N} \vert_{-1}
  \\ & \quad + \vert P \big( 3 (Y_t^{M,N})^2 \mathcal{O}_t^{M,N}
        + 3 Y_t^{M,N} (\mathcal{O}_t^{M,N})^2 + (\mathcal{O}_t^{M,N})^3
        - \mathcal{O}_t^{M,N} \big) \vert_1 \vert \dot{Y}_t^{M,N} \vert_{-1}
  \\ & \leq - \tfrac{\dd }{\dd t} \int_{\mathcal{D}} \Phi(Y_t^{M,N}) \dd x
        + \vert \dot{Y}_t^{M,N} \vert_{-1}^2 + \tfrac 12 \vert Z_t^{M,N} \vert_1^2
  \\ & \quad + C
  \big( 1 + \vert (Y_t^{M,N})^2 \vert_1^2 + \| A Y_t^{M,N} \|^2 \big)
 \big( 1 + \vert \mathcal{O}_t^{M,N} \vert_{\theta}^6 \big).
\end{split}
\end{align*}
Therefore,
\begin{align*}
\begin{split}
J(Y_t^{M,N}) & \leq J( Y_0^{M,N} ) + C \int_0^t \vert Z_s^{M,N} \vert_1^2 \dd s
   \\ & \quad + C \Big( 1 + \int_0^t \big( \vert (Y_s^{M,N})^2 \vert_1^2
                  + \| A Y_s^{M,N} \|^2 \big) \dd s \Big)
      ( 1 + \sup_{s\in[0,T]} \vert \mathcal{O}_s^{M,N} \vert_{\theta}^6 ).
\end{split}
\end{align*}
Applying \eqref{eq:moment-1-A} and Lemma \ref{lem:discrete-stochastic-convolution} infers that
\begin{align*}
\E [ ( \chi_{\Omega_{R_{\tau},t_{i-1}}} J( Y_{t_i}^{M,N} ) )^p ]
   \leq C \Big( 1 + \E \Big[ \int_0^{t_i} \chi_{\Omega_{R_{\tau},t_{i-1}}}
       \vert Z_s^{M,N} \vert_1^2  \dd s \Big]^p \Big).
\end{align*}
Further, we adapt similar arguments used in the proof of \eqref{estimate-XSMN} to get for $s \in [0,t_i)$ and $\kappa \in (\tfrac d2,\text{min}\{\gamma,2\})$,
\begin{equation*}
\begin{split}
\|& \chi_{\Omega_{R_{\tau},t_{i-1}}} X_s^{M,N}
     \|_{L^{p}(\Omega,\dot{H}^\kappa)}
  \\ & \quad \leq C \Big(
        \| X_0\|_{L^{p}(\Omega,\dot{H}^\kappa)}
  + \| \mathcal{O}_s^{M,N} \|_{L^{p}(\Omega,\dot{H}^\kappa)}
  \\ & \qquad + \int_0^s  ( s - r )^{- \frac {2+\kappa}4}
         \Big\| \chi_{ \Omega_{R_{\tau},t_{i-1}} }\| P F ( X_{\lfloor {r} \rfloor_{\tau}}^{M,N}) \| \Big\|_{L^{p}(\Omega,\mathbb{R})} \dd r \Big)
  \\ & \quad \leq C ( 1  + R_{\tau}^3 ).
\end{split}
\end{equation*}
Using \eqref{eq:F'-condition} and the Sobolev embedding inequality $\dot{H}^{\kappa} \subset V$ yields
\begin{equation*}
\begin{split}
\| \chi_{\Omega_{R_{\tau},t_{i-1}}}
  F ( X_{\lfloor {s} \rfloor_{\tau}}^{M,N} )
     \big\|_{L^{p}(\Omega,\dot{H}^1)}
   & = \| \chi_{\Omega_{R_{\tau},t_{i-1}}}
   F ' ( X_{\lfloor {s} \rfloor_{\tau}}^{M,N} )
    \nabla X_{\lfloor {s} \rfloor_{\tau}}^{M,N}
     \big\|_{L^{p}(\Omega,\dot{H})}
   \\ & \leq C \big( 1 + \| \chi_{\Omega_{R_{\tau},t_{i-1}}}
          X_{\lfloor {s} \rfloor_{\tau}}^{M,N}
     \|_{L^{p}(\Omega,\dot{H}^\kappa)}^3 \big)
   \\ & \leq C( 1 + R_{\tau}^9 ).
\end{split}
\end{equation*}
Similarly, employing  \eqref{I-spatial-temporal-S(t)}, \eqref{IV-spatial-temporal-S(t)} and Assumption \ref{assum:intial-value-data} yields
\begin{equation*}
\begin{split}
& \| \chi_{\Omega_{R_{\tau},t_{i-1}}}
   ( X_s^{M,N} - X_{\lfloor {s} \rfloor_{\tau}}^{M,N} )
     \big\|_{L^{p}(\Omega,\dot{H}^1)}
\\ &  \leq C \Big( \tau^{ \frac{\gamma-1}{4} }
   \| X_0 \|_{L^{p}(\Omega,\dot{H}^\gamma)}
    + \tau^{\text{min}\{\frac12,\frac{\gamma-1}4\}}
 \\ & \quad  + \left\| \chi_{\Omega_{R_{\tau},t_{i-1}}} \Big\|
       \int_0^{\lfloor {s} \rfloor_{\tau}} \!\!
         \big( E( s - u ) - E( \lfloor {s} \rfloor_{\tau} - u ) \big)
   \tfrac { A^{\frac32} P_N F (X_{\lfloor {u} \rfloor_{\tau}}^{M,N}) }
         { 1 + \tau \| P_N F (X_{\lfloor {u} \rfloor_{\tau}}^{M,N}) \| } \dd u   \Big\| \right\|_{L^{p}(\Omega,\mathbb{R})}
\\ & \quad +  \left\| \chi_{\Omega_{R_{\tau},t_{i-1}}}
       \int_{\lfloor {s} \rfloor_{\tau}}^s \Big\|
     \tfrac { E(s- u) A^{\frac32} P_N F (X_{\lfloor {u} \rfloor_{\tau}}^{M,N}) }
        { 1 + \tau \|P_N F (X_{\lfloor {u} \rfloor_{\tau}}^{M,N}) \|}
         \Big\|  \dd u \right\|_{L^{p}(\Omega,\mathbb{R})} \Big)
\\ &  \leq C \tau^{\text{min}\{\frac12,\frac{\gamma-1}4\}}
    + C \big( 1 + R_{\tau}^3 \big ) (\tau^{\frac 16}+\tau^{\frac 14})
\\ & \leq C \, \tau^{\min \{ \frac 16, \frac {\gamma-1}4 \}}( 1 + R_{\tau}^3).
\end{split}
\end{equation*}
Finally, it suffices to bound
$ \E \Big[ \chi_{\Omega_{R_{\tau},t_{i-1}}}
   \vert Z_s^{M,N} \vert_1^{p} \Big] $
for $ s \in [0,t_i] $.
Taking $R_\tau = \tau^{- \min \{ \frac 1{54}, \frac {\gamma-1}{36}\}}$ yields
\begin{equation*}
\begin{split}
\| & \chi_{ \Omega_{R_{\tau},t_{i-1}} } Z_s^{M,N} \|_{L^p(\Omega,\dot{H}^1)}
     \leq \big\| \chi_{\Omega_{R_{\tau},t_{i-1}}}
        ( F (X_s^{M,N}) - F (X_{\lfloor {s} \rfloor_{\tau}}^{M,N}) )
            \big\|_{L^p(\Omega,\dot{H}^1)}
 \\ & \quad + \Big\| \chi_{\Omega_{R_{\tau},t_{i-1}}}
   \big( F (X_{\lfloor {s} \rfloor_{\tau}}^{M,N})
     - \tfrac{ F (X_{\lfloor {s} \rfloor_{\tau}}^{M,N}) }
     { 1 + \tau \| P_N F (X_{\lfloor {s} \rfloor_{\tau}}^{M,N}) \|}
    \big) \Big\|_{L^p(\Omega,\dot{H}^1)}
 \\ & \leq C \big\| \chi_{\Omega_{R_{\tau},t_{i-1}}}
   ( X_s^{M,N} - X_{\lfloor {s} \rfloor_{\tau}}^{M,N} )
         \big\|_{L^{2p}(\Omega,\dot{H}^1)}
     \Big( 1 + \| \chi_{\Omega_{R_{\tau},t_{i-1}}}
     X_{\lfloor {s} \rfloor_{\tau}}^{M,N} \|_{L^{4p}(\Omega,\dot{H}^\kappa)}^2
 \\ & \quad  + \| \chi_{\Omega_{R_{\tau},t_{i-1}}} X_s^{M,N}
            \|_{L^{4p}(\Omega,\dot{H}^\kappa)}^2 \big)
 \\ & \quad   + C \tau \| \chi_{\Omega_{R_{\tau},t_{i-1}}} \!\!
        F (X_{\lfloor {s} \rfloor_{\tau}}^{M,N}) \|_{L^{2p}(\Omega,\dot{H}^1)}
        \| \chi_{\Omega_{R_{\tau},t_{i-1}}}       \!\!
        F (X_{\lfloor {s} \rfloor_{\tau}}^{M,N}) \|_{L^{2p}(\Omega,\dot{H})}
 \\ & \leq C \, \tau^{\min \{ \frac 16, \frac {\gamma-1}4 \}}
       ( 1 + R_{\tau}^9) + C \, \tau ( 1 + R_{\tau}^{12})
 \\ & < \infty.
\end{split}
\end{equation*}
The above estimates in combination with the fact
$\vert Y_{t_i}^{M,N} \vert_1^2 \leq 2 J(Y_{t_i}^{M,N}) $ yield
\begin{equation}\label{eq:moment-H1-23d}
\|  \chi_{ \Omega_{R_{\tau},t_{i-1}} } Y_{t_i}^{M,N} \|_{L^p(\Omega,\dot{H}^1)} < \infty,\,\,d=2,3.
\end{equation}
Gathering \eqref{eq:moment-L6-1d}, \eqref{eq:moment-H1-23d}, Sobolev embedding inequality $\dot{H}^1 \subset L^6,d=2,3$ and Lemma \ref{lem:discrete-stochastic-convolution} together completes the proof.
\end{proof}

By adopting similar arguments in
\cite[Theorem 4.6]{Wang2020efficient},
 we can obtain a priori moment bound of
 $\| X_{t_m}^{M,N} \|_{L^6}$ via Markov's inequality.

\begin{theorem}\label{full;a priori estimate}
Under Assumptions \ref{assum:linear-operator-A}--\ref{assum:intial-value-data}, it holds that for any $p \geq 1$,
\begin{equation*}
\sup_{M,N \in \mathbb N^{+}} \sup_{ i \in \{ 0,1,\ldots,M \} }
    \E \Big[ \|X_{t_i}^{M,N}\|_{L^6}^p \Big] < \infty.
\end{equation*}
\end{theorem}
\begin{proof}
By virtue of Lemma \ref{indication;a priori} and the fact that $\Omega_{R_{\tau},t_{i}} \subset \Omega_{R_{\tau},t_{i-1}}$,
it suffices to bound
\begin{equation*}
\sup_{M,N \in \mathbb N^{+}}
\sup_{i \in \{ 0,1,\ldots,M \}}
  \E \Big[ \chi_{\Omega_{R_{\tau},t_{i}}^c}
  \| X_{t_i}^{M,N} \|_{L^6}^p \Big]
\end{equation*}
The case $i=0$ is trivial and we only  consider
$i \in \{1,\ldots,M \}$.
Following a standard argument and using the Sobolev embedding inequality $\dot{H}^1 \subset L^6$ yield
\begin{equation*}
\begin{split}
\|X_{t_i}^{M,N}\|_{L^6}
  &  \leq \|E(t_i)P_N X_0\|_{L^6}
   +\Big\| \int_0^{t_i}
      E(t_i-\lfloor {s} \rfloor_{\tau}) \dd W(s) \Big\|_{L^6}
 \\ & \quad + \Big \|
    \int_0^{t_i} \tfrac{E(t_i-s) A P_N
      F(X_{\lfloor s \rfloor_{\tau}}^{M,N})}
      {1+\tau\|P_N F(X_{\lfloor s \rfloor_{\tau}}^{M,N})\|}
       \dd s \Big\|_{L^6}
 \\ & \leq C \Big( \|X_0\|_{L^6} + \|\mathcal{O}_{t_i}^{M,N}\|_{L^6}
 \\ & \quad + \tfrac{1}{\tau} \int_0^{t_i}
    \|A^{\frac32} E(t_i - s)\|_{\mathcal{L}(H)}
      \tfrac{\tau\|P_N F(X_{\lfloor s\rfloor_{\tau}}^{M,N})\|}
       {1+\tau\|P_N F(X_{\lfloor s\rfloor_{\tau}}^{M,N})\|} \dd s \Big)
\\ & \leq  C \Big( \|X_0\|_{L^6} + \|\mathcal{O}_{t_i}^{M,N}\|_{L^6}
 + \frac{1}{\tau} \int_0^{t_i} (t_i - s)^{-\frac 34} \dd s \Big)
\\ & \leq C \Big( \|X_0\|_{L^6} + \|\mathcal{O}_{t_i}^{M,N}\|_{L^6} + \tau^{-1} \Big).
\end{split}
\end{equation*}
Thanks to Lemma \ref{lem:discrete-stochastic-convolution} and Assumption \ref{assum:intial-value-data}, we have for $p \ge 2$,
\begin{equation}\label{eq:XtiMN}
\|X_{t_i}^{M,N}\|_{L^p(\Omega,L^6)} \le C( 1 + \tau^{-1} ),\,
i \in \{0,1,\ldots,M\}.
\end{equation}
Note that
\begin{equation*}
\Omega_{R_{\tau},t_i}^c
  = \Omega_{R_{\tau},t_{i-1}}^c
    \cup \Big(\Omega_{R_{\tau},t_{i-1}} \cap
      \Big\{ \omega \in \Omega:\|X_{t_i}^{M,N}\|_{L^6}>R_{\tau}\Big\} \Big).
\end{equation*}
Meanwhile, we recall
$\chi_{\Omega_{R_{\tau},t_{-1}}^c}=0$
and then derive
\begin{equation*}
\begin{split}
\chi_{\Omega_{R_{\tau},t_i}^c}
 & =
\chi_{\Omega_{R_{\tau},t_{i-1}}^c}
+
\chi_{\Omega_{R_{\tau},t_{i-1}}}
\cdot \chi_{\{\|X_{t_i}^{M,N}\|_{L^6}>R_{\tau}\}}
\\ &=
\sum_{j=0}^i
\chi_{\Omega_{R_{\tau},t_{j-1}}}
\cdot
\chi_{\{\|X_{t_j}^{M,N}\|_{L^6}>R_{\tau}\}}.
\end{split}
\end{equation*}
Combining \eqref{eq:XtiMN} with Markov's inequality and H\"{o}lder's inequality shows that
\begin{align*}
\begin{split}
\E \Big[ & \chi_{\Omega_{R_{\tau},t_i}^c}
   \| X_{t_i}^{M,N} \|_{L^6}^p \Big]
  \\ & = \sum_{j=0}^i \E \Big[
    \|X_{t_i}^{M,N}\|_{L^6}^p \cdot
     \chi_{\Omega_{R_{\tau},t_{j-1}}} \cdot
      \chi_{\{\|X_{t_j}^{M,N}\|_{L^6}>R_{\tau}\}} \Big]
  \\ & \leq \sum_{j=0}^i \Big(
   \E  \Big[ \|X_{t_i}^{M,N}\|_{L^6}^{2p} \Big] \Big)^{\frac12} \cdot
     \Big( \E \Big[ \chi_{\Omega_{R_{\tau},t_{j-1}}} \cdot
    \chi_{\{\|X_{t_j}^{M,N}\|_{L^6}>R_{\tau}\}} \Big] \Big)^{\frac12}
  \\ & \leq \sum_{j=0}^i C( 1 + \tau^{-p} ) \cdot
     \Big( \mathbb{P} \Big[ \chi_{\Omega_{R_{\tau},t_{j-1}}}
      \|X_{t_j}^{M,N}\|_{L^6} > R_{\tau} \Big] \Big)^{\frac12}
  \\ & \leq C( 1 + \tau^{-p} ) \sum_{j=0}^i \Big(\E \Big[
    \chi_{\Omega_{R_{\tau},t_{j-1}}}
\|X_{t_j}^{M,N}\|_{L^6}^{\tfrac{2(p+1)}{\varrho(\gamma)}}/
(R_{\tau})^{\tfrac{2(p+1)}{\varrho(\gamma)}} \Big] \Big)^{\frac12}
  \\  & \leq C( 1 + \tau^{-p} )
    \sum_{j=0}^i \tau^{p+1} \Big( \E
     \Big[ \chi_{\Omega_{R_{\tau},t_{j-1}}}
      \|X_{t_j}^{M,N}\|_{L^6}^{\frac{2(p+1)}{\varrho(\gamma)}} \Big]
    \Big)^{\frac12} < \infty,
\end{split}
\end{align*}
where $\varrho(\gamma)= \min \{  \frac 4{81}, \frac {\gamma}{24} \}$
 for $d=1$ and
$\varrho(\gamma) =  \min \{ \frac 1{54}, \frac {\gamma-1}{36}\}$ for $d=2,3$.
The proof is now completed.
\end{proof}

With Theorem \ref{full;a priori estimate} at hand,
it is trivial to verify the regularity of $X_t^{M,N}$ in the next corollaries.

\begin{corollary}\label{X_t;gamma bound}
Let Assumptions \ref{assum:linear-operator-A}--\ref{assum:intial-value-data} be fulfilled. Then for any $p \geq 1$, we have
\begin{equation*}
\sup_{M,N \in \N^{+}} \sup_{t\in[0,T]}
 \E \big[ \vert X_t^{M,N} \vert_{\gamma}^p \big] < \infty.
\end{equation*}
\end{corollary}

\begin{proof}
It follows from the
Burkholder--Davis--Gundy--type inequality, \eqref{III-spatial-temporal-S(t)} and \eqref{eq:ass-AQ-condition} that
\begin{equation*}
\begin{split}
\Big\| \int_0^t & E (t-\lfloor {s} \rfloor_{\tau}) P_N \dd W(s)
     \Big\|_{L^p(\Omega,\dot{H}^{\gamma})}
  \\ & \leq  C \Big(\int_0^t \big\| A^{\frac \gamma 2} E (t-\lfloor {s} \rfloor_{\tau}) \big\|_{\mathcal{L}_2^0}^2  \dd s
   \Big)^{\frac 12}
   \\ & = C \Big(\int_0^t \big\| A E (t-\lfloor {s} \rfloor_{\tau})
    A^{\frac {\gamma-2} 2} \big\|_{\mathcal{L}_2^0}^2  \dd s
   \Big)^{\frac 12}
   \\ & \leq C \| A^{\frac {\gamma-2} 2}
               Q^{\frac 12} \|_{\mathcal{L}_2} < \infty.
\end{split}
\end{equation*}
This together with Assumption \ref{assum:intial-value-data} yields that
\begin{equation*}
\begin{split}
\| &
X_t^{M,N}
\|_{L^p(\Omega,\dot{H}^{\gamma})}
\\ &
\leq
\|
E (t)X_0^{M,N}
\|_{L^p(\Omega,\dot{H}^{\gamma})}
+
\Big\|
\int_0^t
\tfrac{E (t-s)A  P_N F (X_{\lfloor {s} \rfloor_{\tau}}^{M,N})}
{1+\tau \| P_N F (X_{\lfloor {s} \rfloor_{\tau}}^{M,N}) \|}
\dd s
\Big\|_{L^p(\Omega,\dot{H}^{\gamma})}
\\
&
\quad
+
\Big\|
\int_0^t
E (t-\lfloor {s} \rfloor_{\tau})
P_N \dd W(s)
\Big\|_{L^p(\Omega,\dot{H}^{\gamma})}
\\
&
\leq
C
+
\Big\|
\int_0^t
\tfrac{E (t-s)A  P_N F (X_{\lfloor {s} \rfloor_{\tau}}^{M,N})}
{1+\tau \| P_N F (X_{\lfloor {s} \rfloor_{\tau}}^{M,N}) \|}
\dd s
\Big\|_{L^p(\Omega,\dot{H}^{\gamma})}.
\end{split}
\end{equation*}
Further, by taking any fixed number $\delta\in(\tfrac32,2)$,
we first consider the case $\gamma\in[1,\delta]$,
\begin{equation*}
\begin{split}
\Big\| & \int_0^t
  \tfrac{E (t-s)A  P_N F (X_{\lfloor {s} \rfloor_{\tau}}^{M,N})}
{1+\tau \| P_N F (X_{\lfloor {s} \rfloor_{\tau}}^{M,N}) \|} \dd s
   \Big\|_{L^p(\Omega,\dot{H}^{\gamma})}
  \\  & \leq \int_0^t \| E(t-s) A^{\frac{\gamma+2}{2}}
           P_N F (X_{\lfloor {s} \rfloor_{\tau}}^{M,N})
               \|_{L^p(\Omega,\dot{H})}  \dd s
  \\  & \leq C \int_0^t
      (t-s)^{-\frac{\gamma+2}{4}} \dd s
    \sup_{i \in \{0,1,\cdots,M\}} \|PF(X_{t_i}^{M,N})\|_{L^p(\Omega,\dot{H})}
  \\ & \leq C \Big( 1 + \sup_{i \in \{0,1,\cdots,M\}}
    \|X_{t_i}^{M,N}\|_{L^{3p}(\Omega,L^6)}^3 \Big) < \infty,
\end{split}
\end{equation*}
where \eqref{I-spatial-temporal-S(t)} and Theorem \ref{full;a priori estimate} were used.
Next, we turn to the  case $\gamma \in (\delta,3)$.
By using the similar approach and the Sobolev embedding inequality $\dot{H}^{\delta} \subset V$, we obtain
\begin{equation}\label{estimate:delta-3}
\begin{split}
\Big\| &
\int_0^t
\tfrac{E (t-s)A  P_N F (X_{\lfloor {s} \rfloor_{\tau}}^{M,N})}
{1+\tau \| P_N F (X_{\lfloor {s} \rfloor_{\tau}}^{M,N}) \|}
\dd s
\Big\|_{L^p(\Omega,\dot{H}^{\gamma})}
\\&
\leq \int_0^t \| E(t-s) A^{\frac{\gamma+1}{2}} A^{\frac12}
 P_N F (X_{\lfloor {s} \rfloor_{\tau}}^{M,N})
 \|_{L^p(\Omega,\dot{H})} \dd s
\\ & \leq C
 \int_0^t (t-s)^{-\frac{\gamma+1}{4}} \dd s
 \sup_{t \in [0,T]} \|PF(X_t^{M,N})\|_{L^p(\Omega,\dot{H}^1)}
\\ & \leq C \Big( 1 +  \sup_{t\in[0,T]}
   \|X_t^{M,N}\|_{L^{3p}(\Omega,\dot{H}^{\delta})}^3 \Big) < \infty.
\end{split}
\end{equation}
Before proceeding further, one uses \eqref{eq:norm-algebra} to derive
\begin{equation*}
\begin{split}
\sup_{t\in[0,T]}
\|PF(X_t^{M,N})\|_{L^p(\Omega,\dot{H}^{2})}
&\leq
C
\sup_{t\in[0,T]}
\|PF(X_t^{M,N})\|_{L^p(\Omega,H^{2}(\mathcal D))}
\\
&
\leq
C
\Big(
1
+
\sup_{t\in[0,T]}
\|X_t^{M,N}\|_{L^{3p}(\Omega,\dot{H}^{2})}^3
\Big)
\\
&
<\infty
\end{split}
\end{equation*}
Bearing this in mind and repeating the same lines of \eqref{estimate:delta-3}
we can prove for $\gamma \in [3,4)$,
\begin{equation*}
\begin{split}
\Big\| &
\int_0^t
\tfrac{E (t-s)A  P_N F (X_{\lfloor {s} \rfloor_{\tau}}^{M,N})}
{1+\tau \| P_N F (X_{\lfloor {s} \rfloor_{\tau}}^{M,N}) \|}
\dd s
\Big\|_{L^p(\Omega,\dot{H}^{\gamma})}
\\&
\leq \int_0^t \| E(t-s) A^{\frac{\gamma}{2}} A
 P_N F (X_{\lfloor {s} \rfloor_{\tau}}^{M,N})
 \|_{L^p(\Omega,\dot{H})} \dd s
\\ & \leq C
 \int_0^t (t-s)^{-\frac{\gamma}{4}} \dd s
 \sup_{t \in [0,T]} \|PF(X_t^{M,N})\|_{L^p(\Omega,\dot{H}^2)}
\\ & \leq C \Big( 1 +  \sup_{t\in[0,T]}
   \|X_t^{M,N}\|_{L^{3p}(\Omega,\dot{H}^{2})}^3 \Big) < \infty.
\end{split}
\end{equation*}
When $\gamma  = 4$,  we only need the boundedness of
$\| X_t^{M,N} \|_{L^{p}(\Omega,\dot{H}^{\frac d2 + 2 +\epsilon})}$
for small enough $\epsilon > 0$ due to the Sobolev embedding theorem.
This is guaranteed by the regularity estimate in $\dot{H}^{\gamma}, \gamma  \in  [3, 4)$.
Thus, the proof is finished.
\end{proof}

\begin{corollary}
Let Assumptions \ref{assum:linear-operator-A}--\ref{assum:intial-value-data} be fulfilled, then for any $p \ge 1$ and $\beta \in [0,\gamma]$,
there exists a constant $C > 0$ such that
\begin{equation*}
\sup_{M,N \in \N^{+}}
    \| X_t^{M,N} - X_s^{M,N} \|_{L^p(\Omega,\dot{H}^{\beta})}
      \leq C \, ( t - s )^{ \min \{ \frac{1}{2},\frac{\gamma-\beta}{4} \} },
      \,\, 0 \leq s < t \leq T.
\end{equation*}
\end{corollary}

\subsection{Strong convergence rate of the fully discrete scheme}
\label{sec;strong convergence}

In this subsection, we are well prepared to analyze the strong convergence rate of the tamed exponential Euler method.

\begin{theorem}[Strong convergence rate of temporal semi-discretization]
\label{theo:strong-time-discretization}
Suppose Assumptions \ref{assum:linear-operator-A}--\ref{assum:intial-value-data}
are valid.
Let $X^N(t)$ and $X_t^{M,N}$ be given by \eqref{eq:space-mild} and
\eqref{full;continuous version}, respectively.
Then for all $ p \geq 1$ we have
\begin{equation*}
\sup_{M,N \in \N^{+}}  \sup_{t \in [0,T] }
\big\| X^N(t) - X_{t}^{M,N} \big\|_{L^p (\Omega , \dot{H})}
   \leq   C \, \tau^{\frac \gamma 4}.
\end{equation*}
\end{theorem}

\begin{proof}
Firstly, we introduce an auxiliary process,
\begin{equation*}
\widetilde{X}_t^{M,N} =  E(t) P_N X_0
   - \int_0^{t} E (t-s) A P_N F (X_s^{M,N}) \dd s
   + \int_0^{t} E (t-s) P_N \dd W(s).
\end{equation*}
According to the uniform moment bounds of $X_s^{M,N}$,
we can follow a standard approach to obtain
$\| \widetilde{X}_t^{M,N} \|_{L^p(\Omega,\dot{H}^{\gamma})} < \infty$ for any
$ t \in [0,T]$.
Then we can separate
$\|X^N(t)-X_t^{M,N}\|_{L^p (\Omega , \dot{H})}$
into two terms:
\begin{equation*}
\|X^N(t)-X_t^{M,N}\|_{L^p(\Omega,\dot{H})} \leq
     \|\widetilde{X}_t^{M,N}-X_t^{M,N}\|_{L^p(\Omega,\dot{H})}
      +  \|X^N(t)-\widetilde{X}_t^{M,N}\|_{L^p(\Omega,\dot{H})}.
\end{equation*}

Next, we split the proof into two parts.
\newline

\noindent $\mathbf{Step 1:}$
Estimate of
$\|\widetilde{X}_t^{M,N}-X_t^{M,N}\|_{L^p(\Omega,\dot{H})}$.
\newline

We decompose the  error $\|\widetilde{X}_t^{M,N}-X_t^{M,N}\|_{L^p(\Omega,\dot{H})}$ into three further parts,
\begin{align*}
\begin{split}
\| &  \widetilde{X}_t^{M,N}  - X_t^{M,N} \|_{L^p(\Omega,\dot{H})}
    \\ & = \Big\| \int_0^t
        \tfrac{E(t-s)A P_N F (X_{\lfloor {s} \rfloor_{\tau}}^{M,N})}
        {1+\tau \| P_N F (X_{\lfloor {s} \rfloor_{\tau}}^{M,N}) \|} \dd s
        - \int_0^t E(t-s) A  P_N F(X_s^{M,N}) \dd s
          \Big\|_{L^p(\Omega,\dot{H})}
    \\ & \quad + \Big\| \int_0^{t} E (t-s) P_N \dd W(s)
         - \int_0^{t} E (t-\lfloor {s} \rfloor_{\tau}) P_N \dd W(s)
                    \Big\|_{L^p(\Omega,\dot{H})}
    \\ & \leq \Big\| \int_0^t E(t-s) A  P_N \big(  F(X_s^{M,N})
             - F(X_{\lfloor {s} \rfloor_{\tau}}^{M,N}) \big) \dd s
              \Big\|_{L^p(\Omega,\dot{H})}
    \\ & \quad + \Big\| \int_0^t E(t-s) A P_N
               F(X_{\lfloor {s} \rfloor_{\tau}}^{M,N}) \dd s - \int_0^t
           \tfrac{E(t-s)A P_N F (X_{\lfloor {s} \rfloor_{\tau}}^{M,N})}
          {1+\tau \| P_N F (X_{\lfloor {s} \rfloor_{\tau}}^{M,N}) \|} \dd s
                 \Big\|_{L^p(\Omega,\dot{H})}
    \\ & \quad + \Big\| \int_0^{t} E (t-s) P_N \dd W(s)
         - \int_0^{t} E (t-\lfloor {s} \rfloor_{\tau}) P_N \dd W(s)
                    \Big\|_{L^p(\Omega,\dot{H})}
    \\ &=: J_1 + J_2 + J_3.
\end{split}
\end{align*}
By using Taylor's formula and mild form of $X_t^{M,N}$,
we divide $J_1$ into four terms,
\begin{equation*}
\begin{split}
J_1 & \leq \Big\| \int_0^t E(t-s) A P_N
           F'( {X}_{\lfloor {s} \rfloor_{\tau}}^{M,N} )
      ( E( s - \lfloor {s} \rfloor_{\tau} ) - I )
             X_{\lfloor {s} \rfloor_{\tau}}^{M,N}  \dd s
           \Big\|_{L^p (\Omega , \dot{H})}
  \\ & \quad + \Big\| \int_0^t E(t-s) A P_N
           F'( X_{\lfloor {s} \rfloor_{\tau}}^{M,N} )
          \int_{\lfloor {s} \rfloor_{\tau}}^s  E ( s - r ) A P_N
           F( X_{\lfloor {r} \rfloor_{\tau}}^{M,N} ) \dd r \dd s
               \Big\|_{L^p (\Omega , \dot{H})}
  \\ & \quad + \Big\| \int_0^t E(t-s) A P_N
           F'( {X}_{\lfloor {s} \rfloor_{\tau}}^{M,N} )
          \int_{\lfloor {s} \rfloor_{\tau}}^s
           E( s - \lfloor {r} \rfloor_{\tau} ) P_N  \dd W(r) \dd s
               \Big\|_{L^p (\Omega , \dot{H})}
  \\ & \quad + \Big\| \int_0^t   E(t-s) A P_N
          \int_0^1  F''( \lambda(X_s^{M,N},
                       X_{\lfloor {s} \rfloor_{\tau}}^{M,N} ) )
  \\ & \qquad \times  ( X_s^{M,N} - X_{\lfloor {s} \rfloor_{\tau}}^{M,N},
          X_s^{M,N} - X_{\lfloor {s} \rfloor_{\tau}}^{M,N} )
        ( 1 - \lambda )  \dd \lambda \dd s
               \Big\|_{L^p (\Omega , \dot{H})}
  \\ & =: J_{11} + J_{12} + J_{13} + J_{14},
\end{split}
\end{equation*}
where $ \lambda ( X_s^{M,N}, X_{\lfloor {s} \rfloor_{\tau}}^{M,N} )
      :=  {X}_{\lfloor {s} \rfloor_{\tau}}^{M,N}
      + \lambda ( X_s^{M,N} - X_{\lfloor {s} \rfloor_{\tau}}^{M,N} )$.

Subsequently, we  treat the above four terms separately.
It follows from \eqref{I-spatial-temporal-S(t)}, \eqref{II-spatial-temporal-S(t)}, \eqref{eq:F'-condition} and
Corollary \ref{X_t;gamma bound} that
\begin{align*}
\begin{split}
J_{11} & \leq \int_0^t \Big\| E( t - s )  A P_N
         F'( X_{\lfloor {s} \rfloor_{\tau}}^{M,N} )
       ( E ( s - {\lfloor {s} \rfloor_{\tau}} ) - I )
             X_{\lfloor {s} \rfloor_{\tau}}^{M,N}
                    \Big\|_{L^p(\Omega,\dot{H})} \dd s
    \\ & \leq C \int_0^t ( t - s )^{-\frac 12}
       \|  F'( X_{\lfloor {s} \rfloor_{\tau}}^{M,N} )
       ( E ( s - {\lfloor {s} \rfloor_{\tau}} ) - I )
             X_{\lfloor {s} \rfloor_{\tau}}^{M,N}
                        \|_{L^{p}(\Omega,\dot{H})} \dd s
    \\ & \leq C \int_0^t ( t - s )^{-\frac 12}
           \big( 1 + \| X_{\lfloor {s} \rfloor_{\tau}}^{M,N}
               \|^2_{L^{4p}(\Omega,\dot{H}^{\gamma})} \big)
    \\ & \qquad  \qquad \qquad \times
           \| ( E( s - {\lfloor {s} \rfloor_{\tau}}) - I )
           A^{- \frac \gamma 2}   A^{\frac \gamma 2}
             X_{\lfloor {s} \rfloor_{\tau}}^{M,N}
              \|_{L^{2p}(\Omega,\dot{H})} \dd s
    \\ & \leq C \, \tau^{\frac \gamma 4}.
\end{split}
\end{align*}
For the error term $J_{12}$, we apply \eqref{I-spatial-temporal-S(t)},
\eqref{eq:F'-condition} and Corollary \ref{X_t;gamma bound} to deduce that
for $\gamma \in (\tfrac d2,2)$,
\begin{align*}
\begin{split}
J_{12} & \leq \int_0^{t} \Big\| E(t-s) A P_N
    F'( {X}_{\lfloor {s} \rfloor_{\tau}}^{M,N} )
     \int_{\lfloor {s} \rfloor_{\tau}}^s E ( s - r ) A P_N
    F ( X_{\lfloor {r} \rfloor_{\tau}}^{M,N} ) \dd r
           \Big\|_{L^p (\Omega,\dot{H})} \dd s
    \\ & \leq \int_0^{t} \int_{\lfloor {s} \rfloor_{\tau}}^s
    ( t - s )^{-\frac 12} \big\| F'(X_{\lfloor {s} \rfloor_{\tau}}^{M,N})
          E( s - r ) A P_N F( X_{\lfloor {r} \rfloor_{\tau}}^{M,N})
           \big\|_{L^{p}(\Omega,\dot{H})} \dd r \, \dd s
    \\ & \leq \int_0^{t} \int_{\lfloor {s} \rfloor_{\tau}}^s
    ( t - s )^{-\frac 12} ( 1 + \| {X}_{\lfloor {s} \rfloor_{\tau}}^{M,N}
           \|^2_{L^{4p}(\Omega,\dot{H}^{\gamma})} )
    \\ & \qquad \qquad \qquad \times ( s - r )^{-\frac12}
           \| P F(X_{\lfloor {r} \rfloor_{\tau}}^{M,N})
              \|_{L^{2p}(\Omega,\dot{H})} \dd r \, \dd s
    \\ & \leq C\, \tau^{\frac 12} \int_0^{t} ( t - s )^{- \frac 12} \dd s
    \big( 1 + \sup_{ s \in [0,T] } \|{X}_{\lfloor {s} \rfloor_{\tau}}^{M,N}
                 \|^2_{L^{4p}(\Omega,\dot{H}^{\gamma})} \big)
    \\ & \qquad \qquad \qquad \times \sup_{ r \in [0,T] }
       \| P F(X_{\lfloor {r} \rfloor_{\tau}}^{M,N})
                            \|_{L^{2p}(\Omega,\dot{H})}
    \\ & \leq C \, \tau^{\frac 12},
\end{split}
\end{align*}
and for $\gamma \in [2,4]$,
\begin{align*}
\begin{split}
J_{12} & \leq \int_0^t \Big\| E(t-s) A P_N
           F'({X}_{\lfloor {s} \rfloor_{\tau}}^{M,N})
              \int_{\lfloor {s} \rfloor_{\tau}}^s
           E( s - r ) A P_N F(X_{\lfloor {r} \rfloor_{\tau}}^{M,N}) \dd r
               \Big\|_{L^p (\Omega , \dot{H})} \dd s
    \\ & \leq \int_0^t \int_{\lfloor {s} \rfloor_{\tau}}^s
      ( t - s )^{-\frac 12} \big\| F'(X_{\lfloor {s} \rfloor_{\tau}}^{M,N})
           E( s - r ) A P_N F( X_{\lfloor {r} \rfloor_{\tau}}^{M,N} )
               \big\|_{L^{p}(\Omega,\dot{H})} \dd r \, \dd s
    \\ & \leq \int_0^t \int_{\lfloor {s} \rfloor_{\tau}}^s
      ( t - s )^{-\frac 12}  ( 1 + \| {X}_{\lfloor {s} \rfloor_{\tau}}^{M,N}
         \|^2_{L^{4p}(\Omega,\dot{H}^\gamma)} )
         \| P F(X_{\lfloor {r} \rfloor_{\tau}}^{M,N})
             \|_{L^{2p}(\Omega,\dot{H}^{2})}  \dd r \, \dd s
    \\ & \leq C\, \tau \int_0^t (t - s )^{-\frac 12} \dd s
    \big( 1 + \sup_{s\in[0,T]} \|{X}_{\lfloor {s} \rfloor_{\tau}}^{M,N}
        \|^2_{L^{4p}(\Omega,\dot{H}^\gamma)} \big)
   \\ & \qquad \qquad \qquad \times \sup_{r \in [0,T]}
     \| P F(X_{\lfloor {r} \rfloor_{\tau}}^{M,N})
      \|_{L^{2p}(\Omega,\dot{H}^2)}
   \\ & \leq C \, \tau.
\end{split}
\end{align*}
Concerning $J_{13}$, we have
\begin{align*}
\begin{split}
J_{13} &= \Big\| \int_0^t E(t - s) A P_N
                 F'({X}_{\lfloor {s} \rfloor_{\tau}}^{M,N})
                 \int_{\lfloor {s} \rfloor_{\tau}}^s
                 E( s - \lfloor {r} \rfloor_{\tau} ) P_N \dd W(r) \, \dd s
               \Big\|_{L^{p}(\Omega,\dot{H})}
   \\  & \leq \Big\| \int_0^{\lfloor {t} \rfloor_{\tau}} E(t - s) A P_N
                F'({X}_{\lfloor {s} \rfloor_{\tau}}^{M,N})
                \int_{\lfloor {s} \rfloor_{\tau}}^s
               E( s - \lfloor {r} \rfloor_{\tau} ) P_N \dd W(r) \, \dd s
               \Big\|_{L^{p}(\Omega,\dot{H})}
   \\  &  \quad + \Big\| \int_{\lfloor {t} \rfloor_{\tau}}^t E(t - s) A P_N     F'({X}_{\lfloor {s} \rfloor_{\tau}}^{M,N})
               \int_{\lfloor {s} \rfloor_{\tau}}^s
               E( s - \lfloor {r} \rfloor_{\tau} ) P_N \dd W(r) \, \dd s
               \Big\|_{L^{p}(\Omega,\dot{H})}
   \\  &  =: J_{131} + J_{132}.
\end{split}
\end{align*}
Without loss of generality, we assume that there exists an integer $m \in \N$
such that $\lfloor {t} \rfloor_{\tau}=t_m$
 and then apply the stochastic Fubini theorem, the Burkholder--Davis--Gundy--type inequality and the H\"{o}lder inequality to obtain
\begin{footnotesize}
\begin{align*}
\begin{split}
J_{131} & = \Big\| \sum_{k=0}^{m-1} \int_{t_{k}}^{t_{k+1}} E(t-s) A P_N
     F'( {X}_{t_k}^{M,N} )  \int_{t_{k}}^{s}
     E( s - \lfloor {r} \rfloor_{\tau} ) \dd W(r) \, \dd s
            \Big\|_{L^{p}(\Omega,\dot{H})}
     \\ & = \Big\| \sum_{k=0}^{m-1} \int_{t_{k}}^{t_{k+1}}
              \int_{t_{k}}^{t_{k+1}} \chi_{[t_k,s)}(r) E(t-s) A P_N
     F'({X}_{t_k}^{M,N}) E( s - \lfloor {r} \rfloor_{\tau} ) \dd s \, \dd W(r)
            \Big\|_{L^{p}(\Omega,\dot{H})}
     \\ & \leq \Big( \sum_{k=0}^{m-1}\!\! \int_{t_{k}}^{t_{k+1}} \!\!
            \Big\| \int_{t_{k}}^{t_{k+1}} \!\!\! \chi_{[t_k,s)}(r) E(t-s) A P_N
     F'({X}_{t_k}^{M,N}) E( s - \lfloor {r} \rfloor_{\tau} ) Q^{\frac12} \dd s
            \Big\|_{L^p(\Omega,\mathcal{L}_2)}^2 \!\!\! \dd r \Big)^{\frac12}
     \\ & \leq C \tau^{\frac12} \Big( \sum_{k=0}^{m-1} \!\!
     \int_{t_{k}}^{t_{k+1}} \!\!\!  \int_{t_k}^{t_{k+1}}\!\! \sum_{l=1}^\infty
            \| E( t - s ) A  F'({X}_{t_k}^{M,N})
         E( s - \lfloor {r} \rfloor_{\tau} ) Q^{\frac12} \eta_l \|_{L^p(\Omega,\dot{H})}^2 \dd s  \, \dd r  \Big)^{\frac12}.
\end{split}
\end{align*}
\end{footnotesize}
Further, by using \eqref{I-spatial-temporal-S(t)}, Lemma \ref{lem:H-1},
Corollary \ref{X_t;gamma bound} and \eqref{eq:ass-AQ-condition},
one can find that for $\gamma \in (1,4]$ and $\kappa =  \tfrac 34, d=1$ and
$\kappa =  1, d=2,3$,
\begin{align*}
\begin{split}
J_{131} & \leq C \, \tau^{\frac12} \Big( \sum_{k=0}^{m-1}
          \int_{t_{k}}^{t_{k+1}}\int_{t_k}^{t_{k+1}} \sum_{l=1}^\infty
          (t - s)^{-\frac {2-\kappa}2}
  \\ & \qquad \times\| A^{\frac \kappa 2}  F'({X}_{t_k}^{M,N})
          E( s - \lfloor {r} \rfloor_{\tau} ) Q^{\frac12} \eta_l \|_{L^p(\Omega,\dot{H})}^2 \dd s \, \dd r \Big)^{\frac12}
    \\  & \leq C \, \tau^{\frac12} \Big( \sum_{k=0}^{m-1}
          \int_{t_{k}}^{t_{k+1}} (t - s)^{-\frac {2-\kappa}2}
          ( 1 + \| {X}_{t_k}^{M,N} \|^4_{L^{2p}(\Omega,\dot{H}^\gamma)})
    \\  & \qquad \qquad \times \int_{t_{k}}^{t_{k+1}} \sum_{l=1}^\infty \|
         A^{\frac12} A^{\frac{2-\gamma}2} E( s - \lfloor {r} \rfloor_{\tau}) A^{\frac{\gamma-2}2} Q^{\frac12}  \eta_l\|^2 \dd r \, \dd s \Big)^{\frac12}
    \\  & \leq C \, \tau^{\frac{4-\max\{3-\gamma,0\}}4}
          \Big( \sum_{k=0}^m \int_{t_k}^{t_{k+1}} (t - s)^{-\frac {2-\kappa}2} \dd s
          \Big)^{\frac12} \|A^{\frac{\gamma-2}2}Q^{\frac12}\|_{\mathcal{L}_2}
    \\  & \leq C \, \tau^{\frac{4-\max\{3-\gamma,0\}}4}.
\end{split}
\end{align*}
For $\gamma \in (\tfrac 12 ,1]$,
from $\lfloor {t} \rfloor_{\tau}=t_m$, \eqref{I-spatial-temporal-S(t)},
\eqref{III-spatial-temporal-S(t)},
\eqref{eq:F'-condition}, \eqref{eq:ass-AQ-condition}, Corollary
\ref{X_t;gamma bound} and
 Burkholder--Davis--Gundy--type inequality, we deduce
\begin{align*}
\begin{split}
J_{131} & = \Big\| \sum_{k=0}^{m-1} \int_{t_{k}}^{t_{k+1}} E(t-s) A P_N
     F'( {X}_{t_k}^{M,N} )  \int_{t_{k}}^{s}
     E( s - \lfloor {r} \rfloor_{\tau} ) \dd W(r) \, \dd s
            \Big\|_{L^{p}(\Omega,\dot{H})}
   \\ & \leq \sum_{k=0}^{m-1} \int_{t_{k}}^{t_{k+1}} ( t -s )^{-\frac12}
         \big( 1 + \| {X}_{t_k}^{M,N} \|_{L^{{\color{blue}4}p}(\Omega,V)}^{{\color{blue}2}} )
         \\ & \qquad \times
         \Big\| \int_{t_{k}}^{s} E( s - \lfloor {r} \rfloor_{\tau} ) \dd W(r) \Big\|_{L^{2p}(\Omega,\dot{H})} \dd s
     \\ & \leq C \sum_{k=0}^{m-1} \int_{t_{k}}^{t_{k+1}}
     ( t -s )^{-\frac12}
     \Big( \int_{t_{k}}^{s}
         \| E( s - \lfloor {r} \rfloor_{\tau} ) \|_{\mathcal{L}_2}^2
         {\color{blue} \dd r}
         \Big)^{\frac 12} \dd s
     \\ & \leq C \, \tau^{\frac \gamma 4}.
\end{split}
\end{align*}
In what follows, we use the same argument to estimate $J_{132}$,
\begin{align*}
\begin{split}
J_{132} & \leq \int_{\lfloor {t} \rfloor_{\tau}}^t
   \Big\| E(t-s) A P_N F'(X_{\lfloor {s} \rfloor_{\tau}}^{M,N})
      \int_{\lfloor {s} \rfloor_{\tau}}^s
      E( s - \lfloor {r} \rfloor_{\tau} ) P_N \dd W(r)
   \Big\|_{L^{p}(\Omega,\dot{H})} \dd s
    \\ & \leq C \int_{\lfloor {t} \rfloor_{\tau}}^t ( t - s )^{-\frac12} \dd s
      \Big( 1 + \sup_{s\in[0,T]}
      \| X_s^{M,N} \|_{L^{2p}(\Omega,\dot{H}^{\gamma})}^2 \Big)
        \tau^{ \frac{\rm{min}\{\gamma,2\}}4 }
    \\ & \leq C \, \tau^{ \frac{ 2 + \rm{min}\{\gamma,2\} } 4 }.
\end{split}
\end{align*}
Owing to the fact $L^1 \subset \dot{H}^{-\delta_0}$ with
$\delta_0 \in ( \tfrac32 , 2 )$ and the regularity of $X_t^{M,N}$,
we obtain
\begin{align*}
\begin{split}
J_{14} & \leq C \int_{0}^t ( t - s )^{-\frac{2+\delta_0}4}
  \Big\| \int_0^1 F''( \lambda (  X_s^{M,N},
          X_{\lfloor {s} \rfloor_{\tau}}^{M,N} ) )
    \\ & \qquad \qquad \times
    (  X_s^{M,N} -  X_{\lfloor {s} \rfloor_{\tau}}^{M,N},
        X_s^{M,N} -  X_{\lfloor {s} \rfloor_{\tau}}^{M,N} )
       ( 1 - \lambda ) \dd \lambda \Big\|_{L^p(\Omega, L^1)} \dd s
    \\ & \leq C \int_0^t \int_0^1 (t-s)^{ -\frac{ 2+\delta_0 } 4 }
        \|  X_s^{M,N} - X_{\lfloor {s} \rfloor_{\tau}}^{M,N}
            \|^2_{L^{4p}(\Omega,\dot{H})}
    \\ & \qquad \qquad \times
        \| \lambda ( X_s^{M,N}, X_{\lfloor {s} \rfloor_{\tau}}^{M,N})
            \|_{L^{2p}(\Omega,V)} \,\dd \lambda\,\dd s
    \\ & \leq C \, \tau^{ \min \{ 1,\frac\gamma2 \} }
          \Big( 1 + \sup_{ s \in [0,T] }
            \|X_s^{M,N}\|_{L^{2p}(\Omega,\dot{H}^\gamma)} \Big)
            \int_0^t (t-s)^{-\frac{2+\delta_0}4} \dd s
    \\ & \leq C \, \tau^{ \min \{ 1,\frac\gamma2 \} }.
\end{split}
\end{align*}
Combining the above estimates together leads to
\begin{equation*}
J_1 \leq C \, \tau^{ \frac {\gamma} 4 }.
\end{equation*}
Due to the regularity of $X_t^{M,N}$ in Corollary \ref{X_t;gamma bound} and properties of nonlinear term $F$, we obtain
\begin{align*}
\begin{split}
J_2 & = \Big\| \int_0^t E(t-s) A  P_N
     F(X_{\lfloor {s} \rfloor_{\tau}}^{M,N})  \dd s
     -  \int_0^t \tfrac{E(t-s)A P_N F (X_{\lfloor {s} \rfloor_{\tau}}^{M,N})}
    {1 + \tau \| P_N F (X_{\lfloor {s} \rfloor_{\tau}}^{M,N}) \|} \dd s
        \Big\|_{L^p(\Omega,\dot{H})}
 \\ & \leq C \, \tau \int_0^t ( t - s )^{-\frac12} \dd s
    \sup_{s \in [0,T]} \| P F(X_s^{M,N}) \|^2_{L^{2p}(\Omega,\dot{H})}
 \\ & \leq C \, \tau.
\end{split}
\end{align*}
It remains to estimate $J_3$ by virtue of \eqref{II-spatial-temporal-S(t)},
\eqref{III-spatial-temporal-S(t)} and \eqref{eq:ass-AQ-condition},
\begin{equation*}
\begin{split}
J_3 & = \Big\| \int_0^{t} E (t-s)
     (I - E ( s-\lfloor {s} \rfloor_{\tau} ))  \dd W(s)
        \Big\|_{L^p(\Omega,\dot{H})}
 \\ & = \Big( \int_0^{t} \big\| A E (t-s) A^{ - \frac {\gamma} 2}
     (I - E ( s-\lfloor {s} \rfloor_{\tau} ))
     A^{ \frac {\gamma-2} 2} Q^{\frac 12} \big\|_{\mathcal{L}_2}^2  \dd s
        \Big)^{\frac12}
 \\ & \leq C \, \tau^{\frac \gamma 4}.
\end{split}
\end{equation*}
Therefore, the estimates of $J_1$, $J_2$ and $J_3$ imply
\begin{equation}\label{eq:order-1}
\| \widetilde{X}_t^{M,N} - X_t^{M,N} \|_{L^p(\Omega,\dot{H})}
   \leq C \, \tau^{\frac \gamma 4}.
\end{equation}
\newline
\noindent$\mathbf{Step 2:}$
Estimate of
$ \| X^N(t)-\widetilde{X}_t^{M,N} \|_{L^p(\Omega,\dot{H})}$.
\newline

For short, by $e^{M,N}(t)$ we denote $X^N(t)-\widetilde{X}_t^{M,N}$,
which satisfies
\begin{equation*}
\tfrac{\dd}{\dd t} e^{M,N}(t) + A^2 e^{M,N}(t)
   = A P_N \big( F(X_{t}^{M,N}) - F(X^N(t)) \big).
\end{equation*}
Multiplying $A^{-1} e^{M,N}(t)$ on both sides and using \eqref{eq:one-side-condition}, \eqref{eq:local-condition} and H\"{o}lder's inequality lead to
\begin{equation*}
\begin{split}
\tfrac12 & \tfrac{\dd}{\dd t} \vert e^{M,N}(t) \vert_{-1}^2
   + \vert e^{M,N}(t) \vert_{1}^2 =
    \langle e^{M,N}(t), F(X_t^{M,N}) - F(X^N(t)) \rangle
    \\ & = \langle e^{M,N}(t) , F (\widetilde{X}_t^{M,N}) - F (X^N(t)) \rangle
       + \langle e^{M,N}(t) , F (X_t^{M,N}) - F (\widetilde{X}_t^{M,N})\rangle
    \\ & \leq \tfrac32 \|e^{M,N}(t)\|^2 + \tfrac12
          \| F (X_t^{M,N}) - F (\widetilde{X}_t^{M,N}) \|^2
    \\ & \leq \tfrac12 \vert e^{M,N}(t) \vert_1^2 + \tfrac98 \vert e^{M,N}(t) \vert_{-1}^2
     \\ & \quad  + C  \| X_t^{M,N} - \widetilde{X}_t^{M,N} \|^2
        \big( 1 + \|X_t^{M,N}\|_V^4 + \|\widetilde{X}_t^{M,N} \|_V^4 \big).
\end{split}
\end{equation*}
Based on the Gronwall inequality and taking expectation, we achieve
\begin{equation*}
\Big\| \int_0^t \vert e^{M,N}(s) \vert_1^2 \dd s \Big\|_{L^p(\Omega,\mathbb{R})}
   \leq C \, \tau^{\frac{\gamma}2},
\end{equation*}
where the regularity of $\widetilde{X}_t^{M,N}$ and $X_t^{M,N}$ and
\eqref{eq:order-1} were used.
We decompose
$\| e^{M,N}(t) \|_{L^p( \Omega , \dot{H})}$ as follows,
\begin{equation*}
\begin{split}
\|e^{M,N}(t)\|_{L^p(\Omega,\dot{H})} & =
   \Big\| \int_0^t E(t-s) A \big( P_N F(X_s^{M,N}) - P_N F(X^N(s)) \big) \dd s
       \Big\|_{L^p(\Omega,\dot{H})}
     \\ & \leq \int_0^t \big\| E(t-s) A \big(  F(X_s^{M,N})
       - F(\widetilde{X}_s^{M,N}) \big) \big\|_{L^p(\Omega,\dot{H})} \dd s
     \\ & \quad + \Big\| \int_0^t E(t-s) A \big( F(\widetilde{X}_s^{M,N})
       - F(X^N(s)) \big)\dd s \Big\|_{L^p(\Omega,\dot{H})}
     \\ & =: K_1 + K_2.
\end{split}
\end{equation*}
Thanks to the regularity of $\widetilde{X}_t^{M,N}$ and $X_t^{M,N}$,
\eqref{I-spatial-temporal-S(t)} and \eqref{eq:order-1},
one can show
\begin{equation*}
\begin{split}
K_1 & \leq \int_0^t (t-s)^{-\frac12}
  \| F(X_s^{M,N}) - F(\widetilde{X}_s^{M,N}) \|_{L^p(\Omega,\dot{H})} \dd s
 \\ & \leq \int_0^t (t-s)^{-\frac12}
    \| X_s^{M,N} - \widetilde{X}_s^{M,N} \|_{L^{2p}(\Omega,\dot{H})}
  \\ & \quad \times \Big( 1 + \|\widetilde{X}_s^{M,N}\|_{L^{4p}(\Omega,V)}^2
            + \|X_s^{M,N}\|_{L^{4p}(\Omega,V)}^2 \Big) \dd s
 \\ & \leq C \, \tau^{\frac \gamma 4}.
\end{split}
\end{equation*}
Resorting to \eqref{I-spatial-temporal-S(t)}, Lemma \ref{lem:H-1}, the regularity of $\widetilde{X}_t^{M,N}$ and $X^{N}(t)$ and H\"{o}lder's inequality, we acquire for $\eta = \min \{ \gamma, \tfrac34 \}$,
\begin{equation*}
\begin{split}
K_2 & \leq C \Big\| \int_0^t (t-s)^{- \frac {2-\eta} {4} }
  \big\| A^{\frac{\eta}{2}} ( F(\widetilde{X}_s^{M,N}) - F(X^N(s)))\big\|\dd s
             \Big\|_{L^p(\Omega,\mathbb{R})}
 \\ & \leq C \Big\| \int_0^t (t-s)^{-\frac{2-\eta}{4}}
 \vert e^{M,N}(s) \vert_1
    \Big( 1 + \vert X^N(s) \vert_{\gamma}^2 + \vert \widetilde{X}_s^{M,N} \vert_{\gamma}^2
      \Big) \dd s \Big\|_{L^p(\Omega,\mathbb{R})}
 \\ & \leq C \Big\| \Big( \int_0^t \vert e^{M,N}(s) \vert_1^2 \dd s \Big)^{\frac12}
 \\ & \qquad \times
    \Big( \int_0^t (t-s)^{-\frac{2-\eta}{2}}
     \big( 1 + \vert X^N(s) \vert_{\gamma}^4 + \vert \widetilde{X}_s^{M,N} \vert_{\gamma}^4
      \big) \dd s \Big)^{\frac12} \Big\|_{L^p(\Omega,\mathbb{R})}
 \\ & \leq C \Big\| \int_0^t \vert e^{M,N}(s) \vert_1^2 \dd s
               \Big\|_{L^p(\Omega,\mathbb{R})}^{\frac12}
 \\ & \leq C \, \tau^{\frac \gamma 4}.
\end{split}
\end{equation*}
Collecting all the estimates obtained so far  finishes the proof.
\end{proof}

At last, gathering Theorem \ref{theo:strong-time-discretization} with Theorem \ref{th:space-rate},
we get the strong convergence rates of the fully discrete scheme \eqref{full;discrete version}.
\begin{corollary}[Strong convergence rates of the full discretization]
\label{coro:strong full discretization}
Let Assumptions \ref{assum:linear-operator-A}--\ref{assum:intial-value-data}
be satisfied.
Then for $p \geq 1$ it holds that
\begin{equation*}
\sup_{ M,N \in \N^{+} } \sup_{ m \in \{0,1,\ldots,M\} }
  \| X(t_m) - X_{t_m}^{M,N} \|_{ L^p (\Omega , \dot{H}) }
  \leq C ( \lambda_N^{-\frac {\gamma} 2} + \tau^{ \frac{\gamma}{4} } ),
  \, \gamma \in (\tfrac d2,4].
\end{equation*}
\end{corollary}

\section{Numerical experiments}\label{sec:numerical-experiments}
In this section, we include some numerical results to confirm the above assertions.
 Consider the following one-dimensional stochastic Cahn--Hilliard equation:
\begin{align}\label{Numerical-example}
\begin{split}
\left\{\begin{array}{lllll}
\tfrac{\partial u}{\partial t}
=\tfrac{\partial^2 w}{\partial x^2} +\dot{W} ,&
(t,x)\in (0, T]\times (0,1),
\\
w = - \tfrac{\partial^2 u}{\partial x^2} + u^3 - u , & x \in (0,1),
\\
\tfrac{\partial u}{\partial x} \big\vert_{x=0}
 = \tfrac{\partial u}{\partial x} \big\vert_{x=1}=0,&t\in(0, T],
 \\
\tfrac{\partial w}{\partial x} \big\vert_{x=0}
 = \tfrac{\partial w}{\partial x} \big\vert_{x=1}=0,&t\in(0, T],
\end{array}\right.
\end{split}
\end{align}
where $\{W(t)\}_{t\in[0,T]}$ is a  $Q$-Wiener process and the
orthonormal eigensystem $ \{ \lambda_j,e_j \}_{j \in \N+}$ of the Neumann Laplacian on $\dot{H}$ is
\begin{equation*}
\lambda_j = j^2 \pi^2, \quad e_j(x) = \sqrt{2} \text{cos}(j \pi x), \quad j \ge 1.
\end{equation*}
Firstly, we approximate \eqref{Numerical-example} by using the (non-tamed) exponential Euler method, given by
\begin{equation}\label{eem}
X_{t_{m+1}}^{M,N}  =  E(\tau) X_{t_m}^{M,N}
%-\tfrac{ A^{-1}(I-E(\tau)) P_N F (X_{t_m}^{M,N})}
%{1+\tau \| P_N F (X_{t_m}^{M,N}) \|}
- \int_{t_m}^{t_{m+1}} E(t_{m+1}-s)AP_NF(X_{t_m}^{M,N}) \dd s
+ E (\tau) P_N \Delta W_m.
\end{equation}
Table \ref{tab:1} shows Monte Carlo simulations of the first moment $\E[ \| X_T^{M,N} \| ]$ of the exponential Euler approximation \eqref{eem} with the initial value $u(0,x)= 20 \sqrt{2}\cos(\pi x), \, x\in (0,1)$ and $N=100$,
where one can observe that $\E[ \| X_T^{M,N} \| ]$ tends to positive infinity rapidly  as $M$ increases.
Here the value `Inf' represents positive infinity  and `NaN' represents `not-a-number' because of an operation `Inf-Inf'.
On the contrary, the tamed exponential Euler method works well
and does not explode for all $M$.
\begin{table}
\centering
\caption{Simulations of the first absolute moment
 $\E[ \| X_T^{M,N} \| ]$ with $ M \in \{ 1,2,\cdots,20 \}$.}
\label{tab:1}       % Give a unique label
\begin{tabular}{llll}
\hline\noalign{\smallskip}
$M$ & $\E[ \| X_T^{M,N} \| ]$ & $M$ & $\E[ \| X_T^{M,N} \| ]$ \\
\noalign{\smallskip}\hline\noalign{\smallskip}
  $M=1$ & 22.2175    & $M=7$  & 3.7510e+73 \\
  $M=2$ & 34.1425    & $M=8$  & Inf \\
  $M=3$ & 132.9797   & $M=9$  & NaN \\
  $M=4$ & 8.1205e+03 & $M=10$  & NaN \\
  $M=5$ & 1.8550e+09 & $\cdots$ & $\cdots$  \\
  $M=6$ & 2.2128e+25 & $M=20$ & NaN \\
\noalign{\smallskip}\hline
\end{tabular}
\end{table}

Next, we will show the convergence rates
of the tamed exponential Euler method as obtained in Theorem \ref{theo:strong-time-discretization}.
For this purpose, we use the fully discrete method \eqref{full;discrete version} to solve \eqref{Numerical-example} with  $u(0,x)= \sqrt{2}\cos(\pi x), \, x\in (0,1)$.
The error bounds are measured in the mean-square sense at the endpoint $T=1$.
Note that the expectations are approximated by computing averages over 1000 samples.
Since the exact solutions are not available at hand, fixing $N=500$, the reference solution is identified with a very small time stepsize $\tau_{exact} = 2^{-16}$.
Four different time stepsizes $\tau=2^{-j}, j=9,10,11,12$ are then used to
carry out the numerical simulations.
%%Moreover, we choose three different kinds of noise and
%%present the resulting mean-square errors in Fig. \ref{F1}--Fig. \ref{F3}, respectively.

\begin{figure}[!htb]
  \centering
  \begin{varwidth}[t]{\textwidth}
  \includegraphics[width=3in]{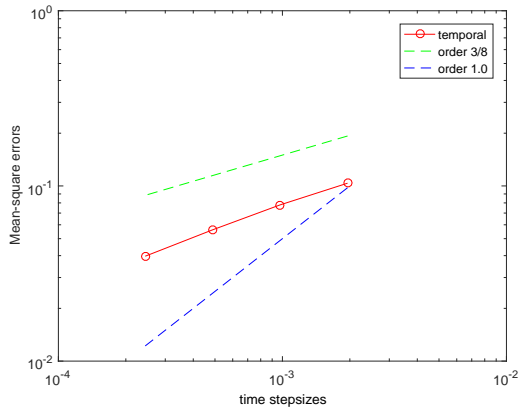}
  \end{varwidth}
  \caption{ Strong convergence rate of the tamed Euler method (white noise).}
 \label{F1}
 \end{figure}

\begin{figure}[!htb]
  \centering
  \begin{varwidth}[t]{\textwidth}
  \includegraphics[width=3in]{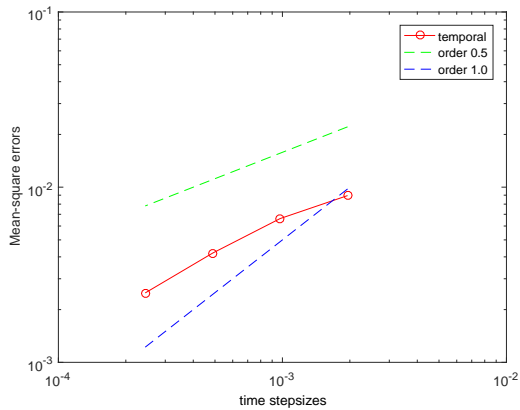}
  \end{varwidth}
  \caption{ Strong convergence rate of the tamed Euler method
     (trace-class noise).}
 \label{F2}
\end{figure}

We are now ready to make some explanations on the numerical results.
For the white noise case (i.e., $Q=I$),
the condition \eqref{eq:ass-AQ-condition} in Assumption \ref{assum:noise-term} is then fulfilled with $\gamma$ closing to $\tfrac32$ and the convergence order obtained in Theorem \ref{theo:strong-time-discretization} is almost $\tfrac38$.
The mean-square errors  are depicted in Fig. \ref{F1}, against $\tau$
on a log-log scale, where one can observe that the resulting numerical errors decrease at a slope close to $\tfrac38$.
This coincides with the theoretical result.
For the trace-class noise case,
we choose $Q$ such that
\begin{equation}\label{Q}
Q e_1=0, \, Q e_i = \frac{1}{i~ \text{log}(i)^2} e_i,\, \forall i \geq 2.
\end{equation}
Obviously, \eqref{Q} guarantees $\text{Tr}(Q)<\infty$ and thus the condition
\eqref{eq:ass-AQ-condition} is satisfied  with $\gamma = 2$.
As expected, the convergence rate of order $\tfrac12$ is detected in
Fig. \ref{F2}, which is consistent with the finding in Theorem \ref{theo:strong-time-discretization}.
For the smoother noise,  $Q$ is then chosen to satisfy
\begin{equation*}
Q e_1=0, \, Q e_i = \frac{1}{i^5~ \text{log}(i)^2} e_i,\, \forall i \geq 2.
\end{equation*}
In this case, condition \eqref{eq:ass-AQ-condition} holds with $\gamma=4$ and the obtained convergence rate in theory is 1.
From Fig. \ref{F3}, it is obvious to find that the approximation errors decrease with order 1, which agrees with the theoretical result.

\begin{figure}[!htb]
  \centering
  \begin{varwidth}[t]{\textwidth}
  \includegraphics[width=3in]{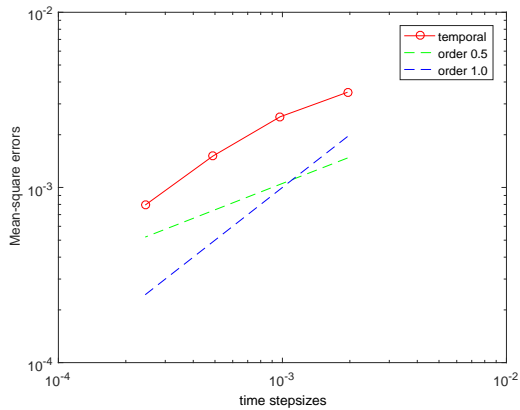}
  \end{varwidth}
  \caption{ Strong convergence rate of the tamed Euler method
     (smoother noise).}
 \label{F3}
\end{figure}

\begin{figure}[!htb]
  \centering
  \begin{varwidth}[t]{\textwidth}
  \includegraphics[width=3in]{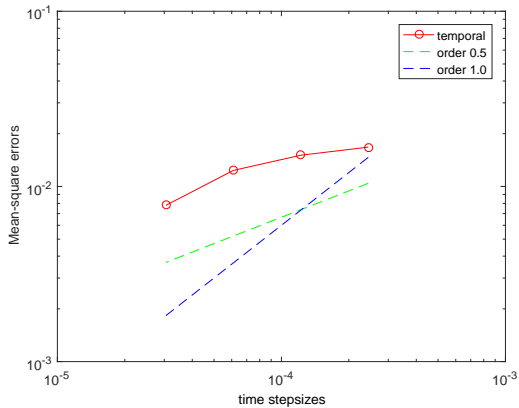}
  \end{varwidth}
  \caption{ Strong convergence rate of the tamed Euler method for $d=2$
     (trace-class noise).}
 \label{F4}
\end{figure}

Moreover, we compare the error  between the tamed exponential Euler method \eqref{full;discrete version} and the backward Euler method.
Based on the simulations over 1000 samples, Table \ref{tab:2}
lists the approximation errors of these two schemes
for five different temporal stepsizes.
Clearly, both schemes give satisfactory accuracy.
However, the backward Euler method needs to
solve a large nonlinear algebraic system by certain iteration
and thus costs more computational efforts than the tamed exponential Euler method.
\begin{table}
\centering
\caption{Comparison of tamed exponential Euler method (TEEM) and  backward Euler method (BEM).}
\label{tab:2}       % Give a unique label
\begin{tabular}{lcc}
\hline\noalign{\smallskip}
Stepsizes &  Error (TEEM) & Error (BEM) \\
\noalign{\smallskip}\hline\noalign{\smallskip}
  $\tau=2^{-9}$  & 0.0195 & 0.0155  \\
  $\tau=2^{-10}$ & 0.0145 & 0.0116  \\
  $\tau=2^{-11}$ & 0.0101 & 0.0083  \\
  $\tau=2^{-12}$ & 0.0069 & 0.0058  \\
  $\tau=2^{-13}$ & 0.0042 & 0.0037  \\
\noalign{\smallskip}\hline
\end{tabular}
\end{table}

\begin{figure}[!htb]
  \centering
  \begin{varwidth}[t]{\textwidth}
  \includegraphics[width=3in]{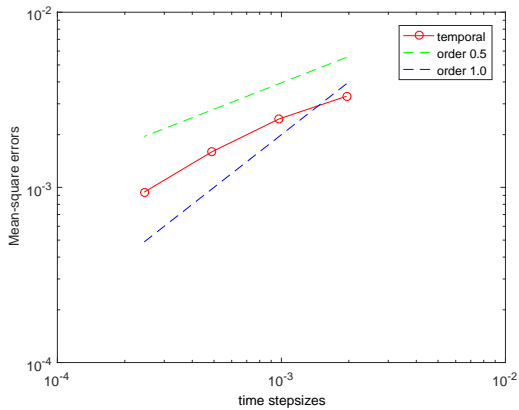}
  \end{varwidth}
  \caption{ Strong convergence rate of the tamed Euler method  without commutative condition.}
 \label{F5}
\end{figure}
For the multi-dimensional case $d=2$, the resulting errors for trace-class noise \eqref{Q} are plotted in Fig. \ref{F4} on a log-log scale, where one can also detect
the expected convergence rate.
Note that the above numerical experiments are performed under the
commutativity of $A$ and $Q$.
Next, by choosing
\begin{equation*}
Q \eta_1=0, \, Q \eta_i = \frac{1}{i~ \text{log}(i)^2} e_i,\,
  \eta_k(x)= \sqrt{2} \text{sin}(k \pi x), \, \forall i \geq 2,\,
  k \ge 1,
\end{equation*}
we attempt to illustrate the error bounds for the fully discrete scheme \eqref{full;discrete version} without the commutative condition of $A$ and $Q$.
From Fig. \ref{F5}, one can observe the expected convergence rate of order $\frac 12$, which agrees with that indicated in Theorem \ref{theo:strong-time-discretization}.
Finally, we mention an interesting circulant embedding approach to the noise sampling recently proposed by \cite{lord2022piecewise}.
We leave it a future work together with some necessary analysis.

\section*{Acknowledgments}
The first author was supported by NSF of China (No. 11971488).
The second author was supported by NSF of China (No. 11701073).
The third author was supported by NSF of China (No. 12071488 and No. 12371417) and NSF of Hunan province (No. 2020JJ2040).
The authors would like to thank the anonymous referees
for valuable comments and suggestions in improving this article.
Great thanks also go to Dr. Xinjie Dai for the help on the numerical experiments.

\section*{Declarations}

{\small { \bf Conflict of interest:} This work does not have any conflicts of interest.}

\end{document}